\documentclass[12pt, dvipsnames, oneside]{amsart}
\usepackage[margin=1.5in]{geometry}                
\usepackage{graphicx}
\usepackage{multicol}
\usepackage{amssymb}
\usepackage[pdfencoding=auto]{hyperref}
\usepackage{cleveref}
\usepackage[backend=biber]{biblatex}
\usepackage{epstopdf}
\usepackage{tikz-cd}
\usepackage{linegoal}
\usepackage{changepage}
\usepackage{xfrac}
\usetikzlibrary{decorations.pathmorphing}
\usepackage{rotating}
\usepackage{biblatex}

\DeclareMathOperator{\Ker}{Ker}
\DeclareMathOperator{\Img}{Im}

\DeclareMathOperator{\Rad}{Rad}

\newtheorem{theorem}{Theorem}[section]

\newtheorem{lemma}[theorem]{Lemma}
\newtheorem{observation}[theorem]{Observation}
\newtheorem{remark}[theorem]{Remark}

\theoremstyle{definition}
\newtheorem{definition}{Definition}[section]
\newtheorem{example}{Example}[section]

\newcommand{\breakcell}[1]{\parbox[t]{\linegoal}{#1}}
\newcommand{\Hom}{\text{Hom}}
\newcommand{\End}{\text{End}}
\newcommand{\bsfrac}[2]{\reflectbox{\sfrac{\text{\reflectbox{$#1$}}}{\text{\reflectbox{$#2$}}}}}

\title{A combinatorial procedure for tilting mutation}
\author{Didrik Fosse}
\newcommand{\Addresses}{{
  \bigskip
  \footnotesize

  Didrik Fosse, \textsc{Department of Mathematical Sciences, NTNU,
    7491 Trondheim, Norway}\par\nopagebreak
}}

\begin{document}

\begin{abstract}
Tilting mutation is a way of producing new tilting complexes from old ones replacing only one indecomposable summand. In this paper, we give a purely combinatorial procedure for performing tilting mutation of suitable algebras.

As an application, we recreate a result due to Ladkani, which states that the path algebra of a quiver shaped like a line (with certain relations) is derived equivalent to the path algebra of a quiver shaped like a rectangle. We will do this by producing an explicit series of tilting mutations going between the two algebras.
\end{abstract}

\maketitle

\section{Introduction}
Tilting theory is central to the study of derived equivalences, and among the most important results in the field is Rickard's Morita theorem for derived categories.\cite{Rickard1989} 

The theorem states that two $k$-algebras are derived equivalent if and only if there exists a certain tilting complex over one of them, which means that for a given $k$-algebra, every tilting complex gives a $k$-algebra which is derived equivalent to it. Thus if we can find a way to easily generate tilting complexes, we get a simple way to generate derived equivalent $k$-algebras. This is where tilting mutation comes in.

Tilting mutation was first developed in \cite{RS1991} as a way of obtaining new tilting complexes by modifying known tilting complexes. This idea was generalised in \cite{BMRRT2004} to define cluster tilting mutation, and in \cite{AI2012} to define silting mutation. 
A purely combinatorial approach to perform silting mutation was developed in \cite{Oppermann2015}. 

\medskip

The goal of this paper is to develop a similar combinatorial approach for tilting mutation, which relies only on modifying the quiver with relations of the algebra we are mutating. Furthermore, we will through an example show how this approach can be used to easily generate a chain of derived equivalent algebras, which in turn could be used as a way to examine whether two given algebras are derived equivalent. This example will be inspired by the work of Ladkani in \cite{Ladkani2013}.

Note that in this paper we only consider right tilting mutation. An entirely dual result can be found by instead using left tilting mutation.
In that case, the combinatorial procedure will be exactly the same, except that all arrows will be flipped around.

\smallskip

\section{Notation}\label{section:setup}

Throughout this paper, $(Q,I)$ will be a quiver with admissible relations, $k$ will be a field, and $\Lambda=kQ/I$ is the corresponding path algebra. For ease of use, we will refer to $(Q,I)$ simply as $Q$. Let $P_i$ denote the indecomposable projective $\Lambda$-module of all paths starting in vertex $i$. 

If there is a path $\beta$ in $Q$ from vertex $i$ to vertex $j$, then there is a morphism $P_j\xrightarrow{\beta} P_i$ given by precomposing with the path $i\xrightarrow{\beta} j$. 
Note that we use the same notation for a path between two vertices and the corresponding morphism between projective modules. This is to keep the notation simple, and it will be clear from context when we are referring to one or the other. Composition of paths in $Q$ will be written from right to left, while composition of morphisms between $\Lambda$-modules will be written from left to right.

For an arrow $\alpha$, we denote by $s(\alpha)$ and $t(\alpha)$ the vertices where the arrow begins and ends, respectively (similarly for paths and relations). 

If $\alpha\colon i\rightarrow j$ is an arrow, and $r\colon i \rightarrow k$ is a linear combination of paths in $Q$ (e.g. a relation), then $\sfrac{r}{\alpha}$ will denote the result of taking the linear combination of all paths in $r$ that begin with the arrow$ \alpha$, and then removing $\alpha$ from each of them. So for each arrow $\alpha$ starting in $i$, there is a linear combination of paths $\sfrac{r}{\alpha}\colon t(\alpha)\rightarrow j$, and together they satisfy $\sum\limits_{\alpha\colon i\rightarrow ?} \sfrac{r}{\alpha}\cdot\alpha = r$ (note that $\sfrac{r}{\alpha}=0$ if $r$ contains no path beginning with $\alpha$). Similarly, $\bsfrac{r}{\beta}\colon h\rightarrow s(\beta)$ refers to the linear combination of paths obtained by removing the arrow $\beta\colon s(\beta)\rightarrow t(r)$ from the end of $r$.

\smallskip

When we write something like $[\beta\alpha]_{\alpha\in A}$ this refers to a vector where each term consists of the expression in brackets, iterating over whatever is in the subscript (in this case it means the vector consisting of each arrow $\alpha\in A$, all followed by the same arrow $\beta$). 

We usually only write $[\beta\alpha]_{\alpha}$, omitting the set we are iterating over, as it is clear from context which set that is.

When we use double subscript we mean the matrix consisting of the expression the first variable along the rows and the second variable down the columns. So for example $[\beta\alpha]_{\alpha,\beta}$ will be a matrix where the $i$-th column consists of the arrow $\alpha_i$ composed with each of the arrows $\beta_j$. 

For the vectors it will be clear from context whether we view them as row or column vectors.

\section{Statement of main theorem}\label{section:mutationsteps}
In this section we state our main result, and give an example to illustrate how it can be used in practice.

The main result of this paper is a purely combinatorial procedure for quiver mutation, which gives us an easy way to perform tilting mutation of a path algebra. Given a quiver $Q$ with relations and a specified vertex $i$ in that quiver, we will construct a new quiver with relations whose path algebra is isomorphic to the algebra obtained by performing tilting mutation of the path algebra of $Q$ at vertex $i$. 

Tilting mutation of $\Lambda$ at $i$ is performed by replacing the indecomposable projective direct summand $P_i$ with another $\Lambda$-module, which we will denote by $P_i^*$. Specifically, $P_i^*$ the cocone of a right approximation of $P_i$ by the other $P_j$ (see \cref{obs:mutationTriangle} for more details).
Inspired by this, the quiver mutation we define here is performed by taking the quiver $Q$ and replacing one of its vertices $i$ with a new vertex $i^*$, and changing the arrows and relations accordingly.

\medskip

Unfortunately, tilting mutation is not always possible, in the sense that mutating an arbitrary tilting complex with respect to an arbitrary indecomposable projective module does not always yield a new tilting complex.

However, under the following assumptions on the path algebra, mutation of $\Lambda$ with respect to $P_i$ is always possible. 

\begin{itemize}
\item Any non-iso endomorphism of $P_i$ factors through another indecomposable projective module. This is equivalent to the quiver having no cycles of length one on vertex $i$ (i.e. no arrows $i\rightarrow i$).
\item We have that $\Hom_{\Lambda}(P_i^*[1], \Lambda) = 0$. This is a sufficient condition for the tilting mutation to work (that is, the result of the mutation will still be a tilting complex).
\end{itemize}

As we will see later, in practice we can use the second point as an easy check to identify some cases where mutation is not possible. For each nonzero path ending in vertex $i$, there must be at least one arrow out of $i$ such that the composition of the path with that arrow is nonzero. If not, then mutation is not possible at vertex $i$. 

In particular, if there is only one arrow $\alpha$ starting in $i$, then mutation is not possible in $i$ if there is a minimal zero relation whose last arrow is $\alpha$. 

\begin{theorem}\label{theorem:mainResult}
Let $\Lambda=kQ/I$ be the path algebra of a quiver $Q$ with relations $I$, and let $i$ be a vertex of $Q$ such that there are no arrows $i\rightarrow i$. Let $\mu_i^R(\Lambda)\simeq\Lambda/P_i\oplus P_i^*$ denote the right tilting mutation of $\Lambda$ at the indecomposable projective $\Lambda$-module $P_i$. We assume that $\Hom_\Lambda(P_i^*[1], \Lambda)=0$ for all $i\neq 0$.
Then, mutating the quiver $Q$ at vertex $i$ according to the mutation procedure stated below yields a quiver, $m_i(Q)$, whose path algebra is isomorphic to $\End_\Lambda(\mu_i^R(\Lambda))^{\text{op}}$.
\end{theorem}

The following list of steps shows exactly which arrows and relations are to be added and removed during the quiver mutation. Note that, rather than simply removing vertex $i$ and adding vertex $i^*$, it can be helpful to visualise the procedure as $i^*$ actually replacing $i$. That way we can talk about ``flipping an arrow'', and ``changing a relation into an arrow'', rather than having to say we remove an arrow/relation to or from vertex $i$ and add an arrow/relation to or from $i^*$.

\emph{The procedure}
\begin{itemize}
\item \textbf{Step 1: Add arrows for compositions through $\boldsymbol{i}$}\\
When we remove vertex $i$ we lose all arrows $\beta\colon h\rightarrow i$ and $\alpha\colon i\rightarrow j$, but the compositions $\alpha\beta\colon h\rightarrow j$ are still paths (which are now minimal). Thus we add them as arrows.\\

\item \textbf{Step 2: Flip arrows out of $\boldsymbol{i}$}\\
Any arrow $\alpha\colon i\rightarrow j$ is replaced by an arrow $\alpha^*\colon j\rightarrow i^*$.\\

\item \textbf{Step 3: Relations out of $\boldsymbol{i}$ become arrows}\\
A minimal relation $r\colon i\dashrightarrow k$ is replaced by an arrow $\overline{r}\colon i^*\rightarrow k$. If the quiver contains a cycle on $i$, and there is a minimal relation $r\colon i\dashrightarrow i$, we get one arrow $\alpha\overline{r}\colon i^*\rightarrow t(\alpha)$ for each arrow $\alpha\colon i \rightarrow t(\alpha)$.\\

\item \textbf{Step 4: Arrows into $\mathbf{i}$ become relations}\\
For each arrow $\beta\colon h\rightarrow i$ we get a new relation $h\dashrightarrow i^*$ given by ${\sum\limits_{\alpha\colon i\rightarrow ?} \alpha^*\alpha\beta=0}$, where $\alpha$ runs over all arrows in $Q$ starting in $i$. \\

\item \textbf{Step 5: Add relations for new compositions}\\
We add a relation for every composition through $i^*$.
Any compositon of arrows through vertex $i^*$ consists of arrows $\alpha^*$ and $\overline{r}$ as defined in step 2 and 3. Each such composition gives rise to a relation given by $\overline{r}\alpha^*=\sfrac{r}{\alpha}$, where $\sfrac{r}{\alpha}\colon t(\alpha)\rightarrow k$ is the (linear combination of) path(s) in the relation $r\colon i\dashrightarrow k$ whose first arrow is $\alpha$, with $\alpha$ removed ($\sfrac{r}{\alpha}=0$ if no such path exists in $r$).\\

\item \textbf{Step 6: Extend relations into $\boldsymbol{i}$}\\
Any relation in $Q$ which ends in $i$ will give relations ending in $t(\alpha)$ for each arrow $\alpha$ starting in $i$. A relation ending in vertex $i$ can be written as $r=\sum\limits_{\beta: ?\rightarrow i} \beta\bsfrac{r}{\beta}$, where $\beta$ runs over all arrows in $Q$ ending in $i$. Thus postcomposing with an arrow $\alpha\colon i \rightarrow j$ gives a relation $\sum\limits_{\beta: ?\rightarrow i} \alpha\beta\bsfrac{r}{\beta}$ ending in $j$.\\

\item \textbf{Step 7: Add relations out of $\boldsymbol{i^*}$}\\
A linear combination of paths from $i^*$ to some vertex $l$ defines a relation in $m_i(Q)$ if and only if precomposing it with each arrow ${\alpha^*\colon t(\alpha)\rightarrow i^*}$ gives a relation in the quiver $Q$. Specifically, if $R$ is a subset of the set of relations in $Q$ which begin in $i$, and $\{\varepsilon_r\}_{r\in R}$ is a collection of (linear combinations of) paths $t(r)\rightarrow l$, then $\sum\limits_{r\in R} \varepsilon_r \overline{r}=0$ is a relation $i^*\dashrightarrow l$ in $m_i(Q)$ if and only if $\sum\limits_{r\in R} \varepsilon_r \sfrac{r}{\alpha}=0$ is a relation in $Q$ for each $\alpha$ starting in $i$. 

\end{itemize}

\medskip

\textbf{Note:} It can happen that we after mutation obtain a relation which is not admissible, i.e. containing a path of length one. Such a relation can be interpreted as setting the corresponding arrow in the quiver equal to some other linear combination of paths (or zero), and thus, removing both the arrow and the relation from the quiver will give an equivalent path algebra. 
In other words, whenever we get a relation which contains a path of length one, we can ``cancel'' the relation against the corresponding arrow, and remove both from the quiver.

\newgeometry{margin=1.1in}

\vspace{2em}

The following table is meant to help visualise how each step of the mutation procedure works in practice. 
It shows the relevant parts of a quiver before and after each step has been applied.
Be aware that these by no means cover all possible cases for how each step behaves. 
The arrows/relations that are directly affected by a given step are marked by green before mutation (only where applicable), and red after mutation.

\vspace{5em}

\begin{tabular}{|c|c|c|l|}
\hline

Step & Before mutation & After mutation & Relations \\ \hline

1 &
\begin{tikzcd}
h \arrow[dr, "\beta"'] & & j\\
& i \arrow[ur, "\alpha"'] &
\end{tikzcd}	
& 
\begin{tikzcd}
h \arrow[r, color=ForestGreen, "{\color{black}\alpha\beta}"] & j
\end{tikzcd} 
& \\ \hline

2&
\begin{tikzcd}
i \arrow[r, color=red, "{\color{black}\alpha}"] & j
\end{tikzcd}	
&
\begin{tikzcd}
i^* & \arrow[l, color=ForestGreen, "{\color{black}\alpha^*}"'] j
\end{tikzcd} 
& \\ \hline

3&
\begin{tikzcd}
i \arrow[dr, "\alpha"'] \arrow[rr, dashed, color=red, "{\color{black}r}"] & & k \\
& j \arrow[ur, "\gamma"'] &
\end{tikzcd}	
&
\begin{tikzcd}
i^* \arrow[rr, color=ForestGreen, "{\color{black}\overline{r}}"] & & k \\
& j \arrow[ul, "{\alpha^*}"] \arrow[ur, "\gamma"'] &
\end{tikzcd}
& \breakcell{$r=\gamma\alpha=0$ \\}  \\ \hline

4&
\begin{tikzcd}
h \arrow[dr, color=red, "{\color{black}\beta}"'] & & j \\
& i \arrow[ur, "\alpha"'] &
\end{tikzcd}	
&
\begin{tikzcd}
h \arrow[rr, "{\alpha\beta}"] \arrow[dr, dashed, color=ForestGreen] & & j \arrow[dl, "{\alpha^*}"] \\
& i^* &
\end{tikzcd} 
& \breakcell{$\alpha^*\alpha\beta=0$ \\} \\ \hline

5&
\begin{tikzcd}
& j_1 & \\
i \arrow[ur, "\alpha_1"] \arrow[r, "\alpha_2"] \arrow[rr, bend right, dashed, "r"'] & j_2 \arrow[r, "\gamma"] & k
\end{tikzcd}	
& 
\begin{tikzcd}
j_1 \arrow[dr, "\alpha_1^*"'] \arrow[drr, dashed, color=ForestGreen] & & \\
& i^* \arrow[r, "\overline{r}" near start] & k\\
j_2 \arrow[ur, "\alpha_2^*"] \arrow[urr, bend right, "\gamma"'] \arrow[urr, dashed, color=red] & &
\end{tikzcd} 
& \breakcell{$r=\gamma\alpha_2=0$ \\ \\ $\overline{r}\alpha_1^*=0$ \\ $\overline{r}\alpha_2^*=\gamma$ \\}\\ \hline

6&
\begin{tikzcd}
& h \arrow[r, "\beta"] &  i \arrow[dr, "\alpha"] &\\
g \arrow[ur, "\delta"] \arrow[urr, dashed] & & & j 
\end{tikzcd}	
& 
\begin{tikzcd}
& h \arrow[r, dashed] \arrow[drr, "\alpha\beta"'] &  i^* &\\
g \arrow[ur, "\delta"] \arrow[rrr, dashed, color=ForestGreen] & & & j \arrow[ul, "\alpha^*"'] 
\end{tikzcd} 
& \breakcell{$\alpha^*\alpha\beta=0$ \\ $\alpha\beta\delta=0$ \\} \\ \hline

7&
\begin{tikzcd}
& j \arrow[r, "\gamma"] \arrow[drr, dashed] &  k  \arrow[dr, "\varepsilon"]&\\
i \arrow[ur, "\alpha"] \arrow[urr, dashed, "r"'] & & & l 
\end{tikzcd} 
&
\begin{tikzcd}
& j \arrow[dl, "\alpha^*"'] &  k \arrow[dr, "\varepsilon"] &\\
i^* \arrow[urr, "\overline{r}"'] \arrow[rrr, dashed, color=ForestGreen] & & & l
\end{tikzcd} 
& \breakcell{$\varepsilon\overline{r}=0$}\\ \hline

\end{tabular}

\restoregeometry

\bigskip

Here are a couple of quick examples to show the steps in action, a more comprehensive example is given in \cref{section:example}. 
\begin{example}
Let $Q$ be the following quiver and let $\Lambda=kQ/I$, where $I$ is the set of relations as indicated by the dashed arrows. 
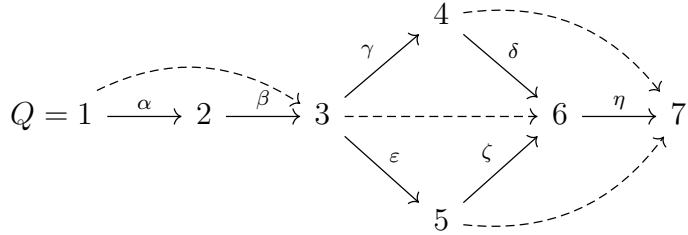
\begin{equation*}
\begin{tikzcd}
 & & & 4 \arrow[dr, "\delta"] \arrow[drr, dashed, bend left] & & \\
 Q = 1 \arrow[r, "\alpha"] \arrow[rr, bend left, dashed] & 2 \arrow[r, "\beta"] & 3 \arrow[ur, "\gamma"] \arrow[dr, "\varepsilon"] \arrow[rr, dashed] & & 6 \arrow[r, "\eta"] & 7 \\
 & & & 5 \arrow[ur, "\zeta"] \arrow[urr, dashed, bend right] & & 
\end{tikzcd}
\end{equation*}
Explicitly, the set of relations is given by
$$\beta\alpha=0,\quad \delta\gamma+\zeta\varepsilon = 0,\quad \eta\delta = 0,\quad \zeta\eta = 0.$$
Suppose we want to use the mutation procedure to determine the resulting quiver from performing right tilting mutation at vertex $3$ in $Q$. Then we simply replace vertex $3$ by $3^*$, and then apply the steps in order, starting with step $1$.
\begin{enumerate}
\item We add arrows $\gamma\beta\colon 2\rightarrow 4$ and $\varepsilon\beta\colon 2\rightarrow 5$, corresponding to the compositions through $3$.
\item We flip the arrows $\gamma\colon 3\rightarrow 4$ and $\varepsilon\colon 3\rightarrow 5$, to get $\gamma^*$ and $\varepsilon^*$.
\item We change the relation $\delta\gamma+\zeta\varepsilon\colon3\dashrightarrow 6$ into an arrow $\overline{\delta\gamma+\zeta\varepsilon}$.
\item We change the arrow $\beta\colon 2\rightarrow 3$ into a relation, defined by $\gamma^*\gamma\beta+\varepsilon^*\varepsilon\beta=0$.
\item We add relations $4\dashrightarrow 6$ and $5\dashrightarrow 6$, corresponding to to the new compositions through $3^*$. Explicitly, they are defined as $\delta+\Big[\overline{\delta\gamma+\zeta\varepsilon}\Big]\gamma^* = 0$ and $\zeta + \Big[\overline{\delta\gamma+\zeta\varepsilon}\Big]\varepsilon^*=0$.
\item We remove the relation $\beta\alpha\colon1\dashrightarrow 3$, and in its place add relations $\gamma\beta\alpha\colon 1\dashrightarrow 4$ and $\epsilon\beta\alpha\colon 1\dashrightarrow 5$.
\item Since there are relations $4\dashrightarrow 7$ and $5\dashrightarrow 7$ we add a relation $3^*\dashrightarrow 7$, given by $\eta\Big[\overline{\delta\gamma+\zeta\varepsilon}\Big]=0$
\end{enumerate} 

Thus we obtain the following quiver

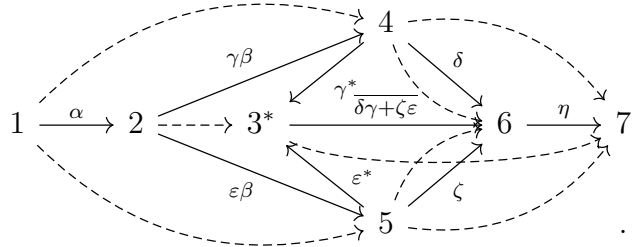
\begin{equation*}
\begin{tikzcd}
 & & & 4 \arrow[dr, "\delta"] \arrow[dl, "{\gamma^*}"] \arrow[dr, bend right, dashed] \arrow[drr, dashed, bend left] & & \\
 1 \arrow[r, "\alpha"] \arrow[urrr, bend left, dashed] \arrow[drrr, bend right, dashed] & 2 \arrow[r, dashed] \arrow[urr, "{\gamma\beta}"] \arrow[drr, "{\varepsilon\beta}"'] & 3^* \arrow[rr, "{\overline{\delta\gamma+\zeta\varepsilon}}"] \arrow[rrr, bend right, looseness=0.5, dashed] & & 6 \arrow[r, "\eta"] & 7 \\
 & & & 5 \arrow[ur, "{\zeta}"'] \arrow[ul, near start, "{\varepsilon^*}"'] \arrow[ur, bend left, dashed] \arrow[urr, dashed, bend right] & & .
\end{tikzcd}
\end{equation*}

The relation $4\dashrightarrow 6$ gives that the arrow $\delta$ is equal (up to sign) to the path $\overline{\delta\gamma+\zeta\varepsilon}$, hence we can remove both the arrow and the relation. The same is true for the arrow $\zeta$ and the relation $5\dashrightarrow 6$. We also remove the relations $4\dashrightarrow 7$ and $5\dashrightarrow 7$, since they both now factor through the relation $3^*\dashrightarrow 7$, and hence are superfluous. Cleaning up a bit, this leaves us with the following quiver as the result of the mutation of $Q$ at vertex $3$.

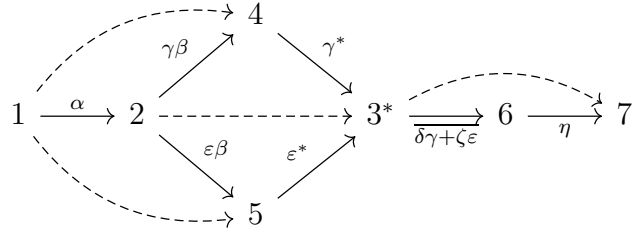
\begin{equation*}
\begin{tikzcd}
 & & 4 \arrow[dr, "{\gamma^*}"] & & & \\
 1 \arrow[r, "\alpha"] \arrow[urr, bend left, dashed] \arrow[drr, bend right, dashed] & 2 \arrow[ur, "{\gamma\beta}"] \arrow[dr, "{\varepsilon\beta}"] \arrow[rr, dashed] & & 3^* \arrow[r, "{\overline{\delta\gamma+\zeta\varepsilon}}"'] \arrow[rr, bend left, dashed] & 6 \arrow[r, "\eta"'] & 7 \\
 & & 5 \arrow[ur, "{\varepsilon^*}"] & & & 
\end{tikzcd}
\end{equation*}

\end{example}

The following example is meant to show that the mutation procedure can be applied to quivers which contain cycles, although the process becomes slightly more complicated. 

\begin{example}

Consider the following quiver, with the relation $\beta\alpha = 0$, and suppose we want to mutate it at vertex $1$.
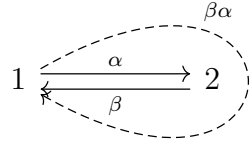
\begin{equation*}
\begin{tikzcd}
1 \arrow[rr, shift left, "\alpha"] \arrow[out=30,in=330,loop,looseness=32, dashed, near start, "{\beta\alpha}"] & & 2 \arrow[ll, shift left, "\beta"] &
\end{tikzcd}
\end{equation*}
We replace vertex $1$ by $1^*$, and then we apply the steps of the mutation procedure in order. Especially take note of step 3, since there is a cycle on vertex $1$.

\begin{enumerate}
\item We add an arrow $\alpha\beta\colon 2\rightarrow 2$, corresponding to the compostition through vertex $1$.
\item We flip the arrow $\alpha$, resulting in the arrow $\alpha^*\colon 2\rightarrow 1^*$.
\item Because the relation $\beta\alpha$ both begins and ends in $1$, we add an arrow $\alpha\overline{\beta\alpha}\colon 1^*\dashrightarrow 2$, corresponding to the composition of the relation $\beta\alpha$ with the arrow $\alpha$.
\item We change the arrow $2\rightarrow 1$ into a relation $2\dashrightarrow 1^*$, given by $\alpha^*\alpha\beta=0$.
\item We add a relation $2\dashrightarrow 2$, corresponding to the new composition through $1^*$. This composition is given by $\alpha\overline{\beta\alpha}\alpha^*=\alpha\beta$
\item We remove the relation $\beta\alpha\colon1\dashrightarrow 1$, and in its place add a relation $1^*\dashrightarrow 2$ given by $\alpha\beta\alpha\overline{\beta\alpha}=0$.
\item Since there is only one minimal relation in the unmutated quiver, step $7$ doesn't come into play.
\end{enumerate}

Thus we obtain the following quiver

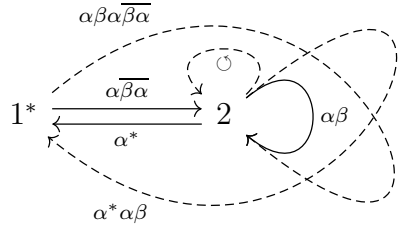
\begin{equation*}
\begin{tikzcd}
1^* \arrow[rr, shift left, "\alpha\overline{\beta\alpha}"] \arrow[rr, out=40,in=320,loop,looseness=5, dashed, very near start, "{\alpha\beta\alpha\overline{\beta\alpha}}"]  & & 2 \arrow[ll, shift left, "\alpha^*"] \arrow[out=40,in=320,loop, "\alpha\beta"] \arrow[ll, out=40,in=320,loop,looseness=5, dashed, very near end, "{\alpha^*\alpha\beta}"] \arrow[out=45,in=135,loop,looseness=5, dashed, "\circlearrowleft"] &
\end{tikzcd}
\end{equation*}

The relation $\alpha\beta = \alpha\overline{\beta\alpha}\alpha^*$ shows that the arrow $\alpha\beta$ factors through the composition $\alpha\overline{\beta\alpha}\alpha^*$, which means that $\alpha\beta$ is a superfluous arrow in the quiver. This also shows that the relation $\alpha^*\alpha\beta = 0$ can be rewritten as $\alpha^*\alpha\beta = \alpha^* (\alpha\overline{\beta\alpha}) = 0$, and the relation $\alpha\beta\alpha\overline{\beta\alpha}$ can be rewritten as $\alpha\overline{\beta\alpha}\alpha^*\alpha\overline{\beta\alpha} = 0$, which is redundant because of the previous relation. Thus we end up with the following quiver, where the relation is given by the composition $1^*\rightarrow 2\rightarrow 1^*$ being zero (this quiver with relation is actually isomorphic to the one we started with).

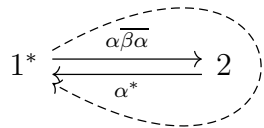
\begin{equation*}
\begin{tikzcd}
1^* \arrow[rr, shift left, "\alpha\overline{\beta\alpha}"] \arrow[out=30,in=330,loop,looseness=25, dashed] & & 2 \arrow[ll, shift left, "\alpha^*"]
\end{tikzcd}
\end{equation*}

\end{example}

This was meant as two simple examples to see how the mutation procedure can be used in practice. 

\section{Tilting mutation}
In this section we will define the right tilting mutation of an algebra, and present some results which we will need in order to prove the main theorem.

\begin{lemma}\label{lemma:minrightapprox}
Let $P_i$ be an indecomposable projective direct summand of $\Lambda=\bigoplus\limits_{j\in Q_0} P_j$, such that $i$ has no cycle of length $1$. Then $\bigoplus\limits_{\alpha:i \rightarrow ?} P_{t(\alpha)} \rightarrow P_i$ is a minimal right $\Lambda/P_i$-approximation of $P_i$ (the sum runs over all arrows starting in vertex $i$).
\end{lemma} 

\bigskip

\begin{observation}\label{obs:mutationTriangle}
Since the derived category of $\Lambda$ is triangulated, we can complete the minimal right approximation $\left[ \bigoplus\limits_{\alpha:i\rightarrow?} P_{t(\alpha)} \xrightarrow{[\alpha]_{\alpha}} P_i \right]$ to a distinguished triangle. We denote by $P_i^*$ the cocone of the morphism $[\alpha]_{\alpha}$, and thus we get a triangle $P_i^*\rightarrow \bigoplus\limits_{\alpha:i\rightarrow?} P_{t(\alpha)} \rightarrow P_i $.
\end{observation}

\bigskip

\begin{definition}
Let $P_i$ be the indecomposable projective direct summand of $\Lambda=kQ/I$ corresponding to vertex $i$. The \emph{right tilting mutation of $\Lambda$ with respect to $P_i$} is defined as $\mu_i^R (\Lambda) :=\Lambda/P_i \oplus P_i^*$.
\end{definition}

\bigskip

One unfortunate property of tilting mutation is that it isn't always possible. When we take $\Lambda$, which is a tilting object, and replace the direct summand $P_i$ by $P_i^*$, it's not guaranteed that the result is again a tilting object.
However, under certain assumptions tilting mutation is always possible, as the following theorem shows.

\bigskip

\begin{theorem}\label{thm:mutationAssumptions}
Let $P_i$ be a projective indecomposable direct summand of $\Lambda=kQ/I$. If $\Hom_\Lambda(P_i^*[1], \Lambda) = 0$, then the tilting mutation $\mu_i^R (\Lambda) = \Lambda/P_i \oplus P_i^*$ is a tilting complex.
\end{theorem}

\begin{proof}
To show that $\Lambda/P_i \oplus P_i^*$ is a tilting complex there are two things we need to check, namely that $\Lambda/P_i \oplus P_i^*$ has no shifted endomorphisms, and that it generates per$(\Lambda)$ as a triangulated category. The module $\Lambda/P_i \oplus P_i^*$ can be viewed as a complex as follows
$$\dots\rightarrow 0 \rightarrow \Lambda/P_i \oplus \bigoplus\limits_{\alpha:i\rightarrow ?} P_{t(\alpha)} \xrightarrow{[0 \ \alpha]_\alpha} P_i \rightarrow 0 \rightarrow \dots$$
Since it only has two nonzero components, any morphism from the complex shifted by two or more in either direction to itself must be zero. The fact that $[\alpha]$ is a right approximation of $P_i$ ensures that zero is the only morphism from the complex shifted by $-1$ to itself. In general, there could be nonzero morphisms from the complex shifted by $1$ to itself, which is why we need the assumption that $\Hom(P_i^*[1], \Lambda)=0$.
 
So given that this assumption is true, $\Lambda/P_i\oplus P_i^*$ will have no shifted endomorphisms. 
To see that it also generates per$(\Lambda)$ as a triangulated category, notice that $\bigoplus_{\alpha}P_{t(\alpha)}$ is a direct summand in $\text{add}(\Lambda/P_i)$, and that it, as well as $P_i$ and $P_i^*$ appear in the triangle
\begin{equation*}
P_i^*\rightarrow \bigoplus\limits_{\alpha: i\rightarrow ?} P_{t(\alpha)} \rightarrow P_i \rightarrow P_i^*[1].
\end{equation*}
This means that $\Lambda/P_i\oplus P_i^*$ and $\Lambda/P_i \oplus P_i\simeq \Lambda$ generate the same triangulated subcategory of the derived category of $\Lambda$. Since $\Lambda$ is a tilting complex, that subcategory is per$(\Lambda)$, which is what we need.
\end{proof}

\bigskip

In the theorem above, the condition is on the algebra $\Lambda$, rather than on the quiver $Q$. 
Still, in some cases it will be enough to look at the quiver in order to conclude that mutation is impossible at a certain vertex.
Clearly, there exists a nonzero morphism ending in $P_i$ if and only if there is at least one arrow starting in vertex $i$ in the quiver. In other words, there is an equivalence
\begin{equation*}
\Hom_\Lambda(\Lambda/P_i, P_i) \neq 0 \iff \exists \alpha \in Q_1 \colon s(\alpha) = i.
\end{equation*}
Notice now that if there exists a nonzero morphism
In other words, condition $1$ is satisfied if and only if there exists at least one arrow in $Q$ starting in vertex $i$. So when we consider a quiver to determine which vertices allow for mutation, we can immediately discard any vertex which has no arrows going out of it.
Condition $2$, on the other hand, is in general not equivalent to a condition on the quiver. It can, however, be used to give an easy criterion on the quiver for when mutation is not possible, as we can see by the following corollary. 

\bigskip

\begin{theorem} 
Let $(Q,I)$ be a quiver with relations, and let $i$ be a vertex in $Q$. Two cases where tilting mutation of $\Lambda=kQ/I$ with respect to $P_i$ is not possible are:
\begin{itemize}
\item If there are no arrows in $Q$ going out of $i$.
\item If there is at least one arrow out of $i$ in $Q$, and there exists a nonzero path in $Q$ ending in $i$ such that composing that path with each arrow out of $i$ gives zero.
\end{itemize}
\end{theorem}

\begin{proof}
If there are no arrows $\alpha$ going out of $i$ then there will be no modules $P_{t(\alpha)}$, and the direct sum $\bigoplus\limits_{\alpha:i \rightarrow ?} P_{t(\alpha)}$ will be equal to zero. This means that there is no nonzero right approximation $\bigoplus\limits_{\alpha:i \rightarrow ?} P_{t(\alpha)}\rightarrow P_i$, so we can't construct $P_i^*$. Thus, mutation is not possible.

Assume now that there is at least one arrow out of $i$, and that there exists a nonzero path $\beta\colon j\rightarrow i$ for some vertex $j$, such that $\alpha\beta = 0$ for each arrow $\alpha\colon i\rightarrow t(\alpha)$. If we view the module $\Lambda/P_i \oplus P_i^*$ as a complex in the same way as before, then this gives us the following diagram.
\begin{equation*}
\begin{tikzcd}
\cdots \arrow[r] & 0 \arrow[d] \arrow[r] & \bigoplus\limits_{\alpha:i\rightarrow ?} P_{t(\alpha)} \arrow[rr, "{[\alpha]_{\alpha}}"] \arrow[d] & & P_i \arrow[d, "{\beta}"] \arrow[r] & 0 \arrow[r] \arrow[d] & \cdots \\
\cdots \arrow[r] & 0 \arrow[r] & 0 \arrow[rr] & & P_j \arrow[r] & 0 \arrow[r] & \cdots 
\end{tikzcd}
\end{equation*}
Notice now that since $\alpha\beta = 0$ for each $\alpha$, the middle square will commute. Thus $\beta$ defines a morphism of complexes between $[\bigoplus_\alpha P_{t(\alpha)}\rightarrow P_i]\simeq P_i^*[1]$ and $P_j$. And composing this with the inclusion of $P_j$ into $\Lambda/P_i$ gives a nonzero morphism $P_i^*[1] \rightarrow \Lambda$, hence $\Hom(P_i^*[1], \Lambda)\neq 0$.
\end{proof}

In particular, if $\alpha$ is the only arrow out of $i$, then mutation is impossible if there is a minimal zero relation whose last arrow is $\alpha$. 
\bigskip

Let's make the following observation about relations in a quiver. Any relation from vertex $i$ to some vertex $k$ is given as some linear combination of paths from $i$ to $k$ being equal to zero. One consequence of this is that, since each of those paths begin with an arrow $i \xrightarrow{\alpha} t(\alpha)$ for some vertex $t(\alpha)$, we can view a relation starting in $i$ as a collection of paths starting in such vertices $t(\alpha)$, satisfying that the path starting in $i$ obtained by precomposing with the arrows $\alpha$  and taking the sum is zero. 

Roughly speaking, this means that a relation starting in $i$ corresponds to a combination of paths starting in the vertices that are hit by arrows from $i$, that become zero if you compose them with the arrows from $i$ to get a collection of paths starting in $i$. We can use this to give a projective resolution of the simple module corresponding to vertex $i$, which will be useful later.

\bigskip

\begin{lemma}\label{lemma:projres}
The sequence
\begin{equation*}
\begin{tikzcd}
\bigoplus\limits_{r:i\dashrightarrow ?} P_{t(r)} \arrow[r, "{[ \sfrac{r}{\alpha} ]_{r,\alpha}}"] & \bigoplus\limits_{\alpha: i\rightarrow ?} P_{t(\alpha)} \arrow[r, "{[\alpha]_\alpha}"] & P_i \arrow[r] & S_i \arrow[r] & 0
\end{tikzcd}
\end{equation*}
defines a projective resolution for $S_i$, the simple $\Lambda$-module at $i$.
\end{lemma}

\begin{proof}
We start by showing that it is in fact a complex. Each component of the map $[ \sfrac{r}{\alpha}]_{r,\alpha}$ is given by the paths $t(\alpha)\rightarrow t(r)$ which are obtained by removing the arrow $\alpha\colon i \rightarrow t(\alpha)$ from the paths $i\rightarrow t(r)$ corresponding to the relations $r\colon i\dashrightarrow t(r)$ passing through $t(\alpha)$. 
For each relation $r$, composing $[\sfrac{r}{\alpha}]_\alpha\colon P_{t(r)}\rightarrow \bigoplus_\alpha P_{t(\alpha)}$ with $[\alpha]_\alpha\colon \bigoplus_\alpha P_{t(\alpha)}\rightarrow P_i$ will yield
$$[\sfrac{r}{\alpha}]_\alpha\cdot[\alpha]_\alpha=\sum_{\alpha:i\rightarrow ?} \sfrac{r}{\alpha}\circ\alpha=r=0.$$
We see that the composition is equal to the relation $r$ (we simply recreate each path in the linear combination by adding the first arrows), which is equal to zero by definition. This is true for all relations $r\colon i\dashrightarrow ?$, hence the composition of $[\sfrac{r}{\alpha}]_{r,\alpha}$ and $[\alpha]_\alpha$ is zero, and we indeed have a complex.
\bigskip

Now, we'll show that the complex is exact. Recall that $P_i$ consists of all paths that start in vertex $i$, and $S_i$ contains only the trivial path in vertex $i$. Since $P_i$ is a projective cover of $S_i$, it is clearly exact in the right-hand term. The kernel of the map $P_i\rightarrow S_i$ will correspond to all paths that start in vertex $i$, except the trivial path. In other words, the kernel consists of all paths starting in vertex $i$ of length $\geq 1$, which is precisely the same as the image of the map $[\alpha]_\alpha$. So the complex is exact in the middle term. Finally, the kernel of $[\alpha]_\alpha$ consists of all paths starting in the $j$-vertices such that precomposing with the arrows $\alpha$ makes (the corresponding linear combination of) them zero. But this is just another way to describe the relations starting in vertex $i$, which is the image of the map $[\sfrac{r}{\alpha}]_{r,\alpha}$. Thus, the complex is exact, which means that the projective terms form a resolution. 
\end{proof}

With these definitions and results, we are ready to start working towards proving the main result of the paper.

\section{Proof of main result}

In this section we will prove that applying the steps given in \cref{section:mutationsteps} to a quiver corresponds to right tilting mutation of its path algebra.
Assume we are given a quiver with relations $(Q,I)$ and a vertex $i$ in $Q$, such that $\Lambda= kQ/I$ satisfies the conditions in \cref{thm:mutationAssumptions}. 
We will show that we can construct a new quiver, $Q_i^{i^*}$, which is isomorphic to the quiver we get when we apply the mutation steps to $(Q,I)$, and whose path algebra is isomorphic to $\End_{\Lambda}(\Lambda/P_i\oplus P_i^*)$.

In order to define the quiver $Q_i^{i^*}$, we must first construct two other quivers, $Q_i$ and $Q^{i^*}$. Let $Q$ be a quiver, let $i$ be a vertex in $Q$ and let $i^*$ not be a vertex in $Q$. Then:
\begin{itemize}
\item $Q_i$ is equal to $Q$ except that vertex $i$ has been removed, and the arrows and relations have been modified so that the path algebra of $Q_i$ is isomorphic to $\End_\Lambda(\Lambda/P_i)$.
\item $Q^{i^*}$ is equal to $Q$ except that a new vertex $i^*$ has been added, along with certain arrows and relations, so that the path algebra of $Q^{i^*}$ is isomorphic to $\End_\Lambda(\Lambda\oplus P_i^*)$.
\end{itemize}

The exact constructions of these two quivers are given in \cref{lemma:QuiverRemoveVertex} and \cref{lemma:Qi*}, respectively.
In terms of these two constructions, the quiver we are after can be given as $Q_i^{i^*}:= (Q^{i^*})_i$, meaning that we first add the vertex $i^*$ to $Q$, and then we remove vertex the $i$ from the resulting quiver. Then, the path algebra of $Q_i^{i^*}$ will be isomorphic to $\End_\Lambda(\Lambda/P_i\oplus P_i^*)$, which is what we want.

What we will do now, is examine the structure of $\End_{\Lambda}(\Lambda/P_i\oplus P_i^*)$ by using the known structure of
$\End_{\Lambda}(\Lambda)\simeq \Lambda$, and we will do this in two steps.
First we show how $\End_\Lambda(\Lambda)$ changes when we add the direct summand $P_i^*$, meaning we look at $\End_{\Lambda}(\Lambda\oplus P_i^*)$. Then we show how $\End_\Lambda(\Lambda\oplus P_i^*)$ changes when we remove the summand $P_i$.
\begin{equation*}
	\begin{tikzcd}
		\End_{\Lambda}(\Lambda) \arrow[rr, dotted] \arrow[dr, "{\text{Add } P_i^*}"', sloped] & & \End_{\Lambda}(\Lambda/P_i\oplus P_i^*) \\
		& \End_{\Lambda}(\Lambda \oplus P_i^*) \arrow[ur, "{\text{Remove } P_i}"', sloped]& 
	\end{tikzcd}
\end{equation*}

The first step corresponds to adding a new vertex $i^*$ to the quiver $Q$, and changing/adding arrows and relations in accordance with $\End_{\Lambda}(\Lambda\oplus P_i^*)$. We denote this quiver by $Q^{i^*}$. The next step corresponds to removing the vertex $i$, along with any arrow and relation starting or ending in it (but keeping compositions through $i$). We will denote this quiver by $Q_i^{i^*}$.

In practice what we will do is use the mutation triangle defined in \cref{obs:mutationTriangle}, together with the known structure of $\End_\Lambda(\Lambda)$ to determine the structure of $\End_\Lambda(\Lambda\oplus P_i^*)$, and then we'll show how the resulting path algebra is affected by the removal of vertex $i$, giving us $\Lambda/P_i\oplus P_i^*$. For this last part we will use the following lemma, which states what happens to a path algebra when we remove a vertex from it. 

\begin{lemma}\label{lemma:QuiverRemoveVertex}
Let $\Lambda = kQ/I$ be a path algebra. Let $i$ be a vertex in $Q$ with no cycles of length one, and with no minimal relations from $i$ to $i$. The quiver associated to $\mathrm{End}_\Lambda(\Lambda/P_i)\simeq (1-e_i)\Lambda(1-e_i)$ is obtained by removing vertex $i$ from the quiver $Q$. We will denote this quiver by $Q_i$, and its arrows and relations are explicitly given by the following:
\begin{itemize}
\item Any arrow in $Q$ that neither starts nor ends in vertex $i$ is also an arrow in $Q_i$. The same is true for relations.
\item If there are arrows $\beta\colon h\rightarrow i$ and $\alpha\colon i\rightarrow j$ in $Q$, then there is an arrow $\alpha\beta\colon h\rightarrow j$ in $Q_i$.
\item If there is an arrow $\alpha\colon i\rightarrow j$ and a relation $r\colon g\dashrightarrow i$ in $Q$, then $\alpha r\colon g\dashrightarrow j$ is a relation in $Q_i$.
\item If there is an arrow $\beta\colon h\rightarrow i$ and a relation $r'\colon i\dashrightarrow k$ in $Q$, then $r'\beta\colon h\dashrightarrow k$ is a relation in $Q_i$.
\end{itemize}
\end{lemma}

\textbf{Note:} There is no ambiguity when we denote arrows and relations in $Q_i$ as the compositions of arrows/relations in $Q$, because any time we do so those arrows/relations are not themselves in $Q_i$.

\begin{proof}[Proof of lemma]
Because the only paths in $Q$ which are not paths in $Q_i$ are those that begin or end in vertex $i$, any path that does neither will be preserved. And since arrows are paths of length one, all arrows that neither begin nor end in vertex $i$ will still be arrows in $Q_i$. This also applies to relations. 
Assume now that we have arrows $\beta\colon h\rightarrow i$ and $\alpha\colon i\rightarrow j$ in $Q$. The composition $\alpha\beta$ is a path from $h$ to $j$, and will therefore be preserved. However, since neither $\alpha$ nor $\beta$ will appear in $Q_i$, we can't decompose $\alpha\beta$ in $Q_i$. This means that the length of $\alpha\beta$ is $1$ in $Q_i$, so it is an arrow. Similarly, for a relation into (resp. out of) vertex $i$, postcomposing with an arrow out of $i$ (resp. precomposing with an arrow into $i$) will give a relation which is preserved in $Q_i$. 

\end{proof}

\bigskip

We will now turn our attention to the first part of the proof of the main result, namely determining how the quiver $Q$ changes when we add the vertex $i^*$ to it. 
Combined with the above lemma, this will allow us to prove the main result.

\bigskip

\subsection*{The structure of \texorpdfstring{$\End(\Lambda\oplus P_i^*)$}{End(Λ⊕Pi*)}}
We begin by considering $\End_\Lambda(\Lambda\oplus P_i^*)$. We want to find a generating subset (which will correspond to the arrows in the quiver $Q^{i^*}$). Notice that we can split this into four direct summands, which can be considered individually:
$$\End_\Lambda(\Lambda\oplus P_i^*) \simeq \begin{bmatrix}\End_{\Lambda}(\Lambda)  & \Hom_{\Lambda}(\Lambda, P_i^*) \\ \Hom_{\Lambda}(P_i^*, \Lambda) & \End_{\Lambda}(P_i^*)\end{bmatrix}.$$ 
We already know the structure of $\End_\Lambda(\Lambda)$, so we only need to examine the other three direct summands. The sets $\Hom_\Lambda(\Lambda, P_i^*)$ and $\Hom_\Lambda(P_i^*,\Lambda)$ correspond to paths out of and into $i^*$ in the quiver $Q^{i^*}$, respectively (recall that a path $i\rightarrow j$ gives a map $P_i\leftarrow P_j$). $\End_{\Lambda}(P_i^*)$ corresponds to paths that both start and end in $i^*$. This always contains the trivial path $e_i^*$, but could also contain cycles of length $\geq 1$.

Recall that $P_i^*$ is defined by completing the right $\Lambda/P_i$-approximation $\bigoplus\limits_{\alpha: i\rightarrow ?} P_{t(\alpha)}\rightarrow P_i$ of $P_i$ to the distinguished triangle
$$P_i^*\xrightarrow{[\alpha^*]_\alpha} \bigoplus\limits_{\alpha:i\rightarrow ?} P_{t(\alpha)}\xrightarrow{[\alpha]_\alpha} P_i.$$

Throughout this proof, whenever we apply the functor $\Hom_{\Lambda}(\Lambda, - )$ to a map $P_j\xrightarrow{\alpha} P_i$, we will denote the induced map as $\alpha$ as well (same for the corresponding contravariant $\Hom$-functor). This abuse of notation is to make the proof easier to read. We will also write $\Hom_{\Lambda}(-,-)$ as $(-,-)$, to save space. 
\bigskip\\

\begin{lemma}\label{lemma:EndPi*}
Any morphism from $P_i^*$ to itself which is not an isomorphism will factor through $P_i^*\xrightarrow{[\alpha^*]_\alpha}\bigoplus_\alpha P_{t(\alpha)}$.
\end{lemma}

\begin{proof}

Assume that $\gamma$ is a morphism from $P_i^*$ to itself which is not an isomorphism. We will now show that this can be completed to a morphism of distinguished triangles from the mutation triangle to itself. Consider the solid part of the following diagram.
\begin{equation*}
\begin{tikzcd}
P_i[-1] \arrow[r, "c"] \arrow[d, dashed, "{\beta[-1]}"] & P_i^* \arrow[r, "{[\alpha^*]_{\alpha}}"] \arrow[d, "{\gamma}"] & \bigoplus\limits_{\alpha:i\rightarrow ?}P_{t(\alpha)}  \arrow[r, "{[\alpha]_{\alpha}}"] \arrow[d, dotted, "{\delta}"] & P_i \arrow[d, dashed, "{\beta}"] \\
P_i[-1] \arrow[r, "c"] & P_i^* \arrow[r, "{[\alpha^*]_{\alpha}}"]  & \bigoplus\limits_{\alpha:i\rightarrow ?} P_{t(\alpha)}  \arrow[r, "{[\alpha]_{\alpha}}"] & P_i
\end{tikzcd}
\end{equation*}
The composition $c \gamma\cdot[\alpha^*]_{\alpha}$ is a morphism between projective objects in different shifted degrees, and is therefore zero. Since $P_i[-1]$ is a weak kernel of $[\alpha^*]_\alpha$, this implies that $c \gamma $ factors through $P_i[-1]$, and thus there exists a morphism $\beta[-1]$ making the square commute. The $2$-out-of-$3$ property for triangulated categories now ensures the existence of $\delta$, completing the morphism of triangles.  

\bigskip

We will now show that $\beta$ can't be an isomorphism. If $\beta$ is an isomorphism then it has an inverse $\beta^{-1}\colon P_i\rightarrow P_i$.
The morphism $\bigoplus_{\alpha}P_{t(\alpha)} \xrightarrow{[\alpha]_{\alpha}} P_i$ is a right $\Lambda/P_i$-approximation, so the composition $[\alpha]_\alpha\cdot\beta^{-1}$ factors through $\bigoplus_{\alpha}P_{t(\alpha)}$. Thus there exists a morphism $\delta'$ making the following diagram commute.
\begin{equation*}
\begin{tikzcd}
\bigoplus\limits_{\alpha:i\rightarrow ?}P_{t(\alpha)}  \arrow[r, "{[\alpha]_{\alpha}}"] \arrow[d, "{\delta}"] & P_i \arrow[d, "{\beta}"] \\
\bigoplus\limits_{\alpha:i\rightarrow ?}P_{t(\alpha)}  \arrow[r, "{[\alpha]_{\alpha}}"] \arrow[d, dashed, "{\delta'}"] & P_i \arrow[d, "{\beta^{-1}}"] \\
\bigoplus\limits_{\alpha:i\rightarrow ?} P_{t(\alpha)}  \arrow[r, "{[\alpha]_{\alpha}}"] & P_i
\end{tikzcd}
\end{equation*}
We now have that $\delta\delta'\cdot[\alpha]_\alpha=[\alpha]_\alpha\cdot\beta^{-1}\beta=[\alpha]_\alpha$, and since $[\alpha]_\alpha$ is right minimal, this implies that $\delta\delta'$ is an isomorphism. 
This means that $\delta$ has a right inverse, explicitly given by $\left[\delta'(\delta\delta')^{-1}\right]$. 
A similar argument shows that $\delta$ has a left inverse, and hence it is an isomorphism. In other words, $\beta$ being an isomorphism implies $\delta$ is an isomorphism. But then, by the five-lemma, $\gamma$ would also be an isomorphism, which contradicts our assumption. So $\beta$ can't be an isomorphism.

\bigskip

This means that $\beta\colon i\rightarrow i$ is a linear combination of paths whose length is bigger than $1$, so it factors through $\bigoplus_{\alpha}P_{t(\alpha)} $, since the direct sum is taken over all arrows $\alpha$ that start in $i$ (and we assume the quiver has no cycles of length one at the vertex i). Using this, we will now show that $\gamma$ factors through $[\alpha^*]_{\alpha}$.
For each $\alpha$, let $d_\alpha$ be a path $t(\alpha)\rightarrow i$ such that $\beta=\sum_\alpha \alpha d_\alpha$.  This gives a morphism $[d_\alpha]_\alpha\colon P_i\rightarrow \bigoplus_\alpha P_{t(\alpha)}$ satisfying $\beta=[\alpha]_{\alpha}\cdot [d_\alpha]_\alpha$. Inserting this into the diagram from before we get the following, where the lower right triangle commutes, in addition to the squares.

\begin{equation*}
\begin{tikzcd}
P_i[-1] \arrow[r, "c"] \arrow[d, "{\beta[-1]}"] & P_i^* \arrow[r, "{[\alpha^*]_{\alpha}}"] \arrow[d, "{\gamma}"] & \bigoplus\limits_{\alpha:i\rightarrow ?}P_{t(\alpha)}  \arrow[r, "{[\alpha]_{\alpha}}"] \arrow[d, "{\delta}"] \arrow[dr, phantom, "\circlearrowleft", near end ] & P_i \arrow[d, "{\beta}"] \arrow[dl, pos=0.4, "{[d_\alpha]_\alpha}"] \\
P_i[-1] \arrow[r, "c"'] & P_i^* \arrow[r, "{[\alpha^*]_{\alpha}}"']  & \bigoplus\limits_{\alpha:i\rightarrow ?} P_{t(\alpha)}  \arrow[r, "{[\alpha]_{\alpha}}"'] & P_i
\end{tikzcd}
\end{equation*}

Commutativity of the rightmost square and triangle gives that 
$$\delta\cdot[\alpha]_{\alpha}=[\alpha]_{\alpha}\cdot\beta=[\alpha]_{\alpha}\cdot [d_\alpha]_\alpha\cdot[\alpha]_{\alpha},$$ 
and by rearranging this, we get $\left( \delta - [\alpha]_\alpha\cdot [d_\alpha]_\alpha \right)\cdot [\alpha]_\alpha=0$. Thus, since $P_i^*$ is a weak kernel of $[\alpha]_{\alpha}$ there exists a morphism $[d'_\alpha]_\alpha\colon \bigoplus_{\alpha}P_{t(\alpha)} \rightarrow P_i^*$ such that $[d'_\alpha]_\alpha\cdot [\alpha^*]_{\alpha}= \delta - [\alpha]_\alpha\cdot [d_\alpha]_\alpha$. We insert $[d'_\alpha]_\alpha$ into the diagram

\begin{equation*}
\begin{tikzcd}
P_i[-1] \arrow[r, "c"] \arrow[d, "{\beta[-1]}"] & P_i^* \arrow[r, "{[\alpha^*]_{\alpha}}"] \arrow[d, "{\gamma}"]  & \bigoplus\limits_{\alpha:i\rightarrow ?}P_{t(\alpha)}  \arrow[r, "{[\alpha]_{\alpha}}"] \arrow[d, "{\delta}"] \arrow[dl, dashed, "{[d'_\alpha]_\alpha}"] \arrow[dr, phantom, "\circlearrowleft", near end ] & P_i \arrow[d, "{\beta}"] \arrow[dl, pos=0.4, "{[d_\alpha]_\alpha}"] \\
P_i[-1] \arrow[r, "c"'] & P_i^* \arrow[r, "{[\alpha^*]_{\alpha}}"']  & \bigoplus\limits_{\alpha:i\rightarrow ?} P_{t(\alpha)}  \arrow[r, "{[\alpha]_{\alpha}}"'] & P_i
\end{tikzcd}
\end{equation*}

Observe that we can rewrite the last equation as $\delta=[d'_\alpha]_\alpha\cdot[\alpha^*]_{\alpha} + [\alpha]_\alpha\cdot[d_\alpha]_\alpha$, so by the commutativity of the middle square we get 

\begin{align*}
\gamma\cdot[\alpha^*]_{\alpha} &= [\alpha^*]_{\alpha}\cdot\gamma \\
&=  [\alpha^*]_{\alpha}\cdot\left([d'_\alpha]_\alpha\cdot[\alpha^*]_{\alpha} +  [\alpha]_{\alpha}\cdot[d_\alpha]_\alpha\right) \\
&= [\alpha^*]_{\alpha}\cdot [d'_\alpha]_\alpha\cdot[\alpha^*]_{\alpha} + \underbrace{[\alpha^*]_{\alpha}\cdot[\alpha]_{\alpha}}_{0}\cdot[d_\alpha]_\alpha = [\alpha^*]_{\alpha}\cdot [d'_\alpha]_\alpha\cdot[\alpha^*]_{\alpha}
\end{align*}

This means that $\left( \gamma - [\alpha^*]_{\alpha}\cdot[d'_\alpha]_\alpha\right)\cdot[\alpha^*]_{\alpha}=0$, and since $P_i[-1]$ is a weak kernel of $[\alpha^*]_{\alpha}$ there exists a morphism $d''\colon P_i^*\rightarrow P_i[-1]$ such that $d'' c=\gamma - [\alpha^*]_{\alpha}\cdot [d'_\alpha]_\alpha$. Inserting $d''$ into the diagram, we end up with

\begin{equation*}
\begin{tikzcd}
P_i[-1] \arrow[r, "c"] \arrow[d, "{\beta[-1]}"] & P_i^* \arrow[r, "{[\alpha^*]_{\alpha}}"] \arrow[d, "{\gamma}"] \arrow[dl, dotted, "{d''}"] & \bigoplus\limits_{\alpha:i\rightarrow ?}P_{t(\alpha)}  \arrow[r, "{[\alpha]_{\alpha}}"] \arrow[d, "{\delta}"] \arrow[dl, dashed, "{[d'_\alpha]_\alpha}"] \arrow[dr, phantom, "\circlearrowleft", near end ] & P_i \arrow[d, "{\beta}"] \arrow[dl, pos=0.4, "{[d_\alpha]_\alpha}"] \\
P_i[-1] \arrow[r, "c"'] & P_i^* \arrow[r, "{[\alpha^*]_{\alpha}}"']  & \bigoplus\limits_{\alpha:i\rightarrow ?} P_{t(\alpha)}  \arrow[r, "{[\alpha]_{\alpha}}"'] & P_i
\end{tikzcd}
\end{equation*}

Again we can rewrite the last equation into $\gamma=d'' c + [\alpha^*]_{\alpha}\cdot [d'_\alpha]_\alpha$. Thus, if we can show that $d''=0$, we see that $\gamma$ factors through $\bigoplus_{\alpha}P_{t(\alpha)}$. Note that $d''$ is an element of $(P_i^*, P_i[-1])\simeq(P_i^*[1], P_i)$, which is zero by the assumption that $(P_i^*[1], \Lambda)=0$.

 Thus, we have that $d''=0$, which means that $\gamma=[\alpha^*]_{\alpha}\cdot [d'_\alpha]_\alpha$, so $\gamma$ factors through $\bigoplus_\alpha P_{t(\alpha)}$. This is true for any non-isomorphism $P_i^*\rightarrow P_i^*$.
 
 \end{proof}

\begin{lemma}\label{lemma:HomPiLambda}
As a right $\End_\Lambda(\Lambda \oplus P_i^*)$-module, $(P_i^*, \Lambda \oplus P_i^*)$ fits in the following exact sequence
\begin{equation*}
\begin{tikzcd}[column sep = small]
(P_i,\Lambda \oplus P_i^*) \arrow[r, "{[\alpha]_{\alpha}}"] & \bigoplus\limits_{\alpha:i\rightarrow ?}(P_{t(\alpha)},\Lambda \oplus P_i^*) \arrow[r, "{[\alpha^*]_{\alpha}}"] & (P_i^*,\Lambda \oplus P_i^*) \arrow[r] & \frac{(P_i^*,\Lambda\oplus P_i^*)}{\Rad((P_i^*,\Lambda\oplus P_i^*))} \arrow[r] & 0,
\end{tikzcd}
\end{equation*}
where $\frac{(P_i^*,\Lambda\oplus P_i^*)}{\Rad((P_i^*,\Lambda\oplus P_i^*))}\simeq S_i^*$, the simple module at vertex $i^*$.
\end{lemma}

\begin{proof}
If we apply $(-,\Lambda \oplus P_i^*)$ to the triangle
\begin{equation*}
P_i^*\xrightarrow{[\alpha^*]_{\alpha}} \bigoplus\limits_{\alpha: i\rightarrow ?} P_{t(\alpha)} \xrightarrow{[\alpha]_{\alpha}} P_i \rightarrow P_i^*[1],
\end{equation*}
we get a long exact sequence, which contains the following part 
\begin{equation*}
\begin{tikzcd}[column sep = small]
\cdots \arrow[r] & (P_i,\Lambda \oplus P_i^*) \arrow[r, "{[\alpha]_{\alpha}}"] & \bigoplus\limits_{\alpha:i\rightarrow ?}(P_{t(\alpha)},\Lambda \oplus P_i^*) \arrow[r, "{[\alpha^*]_{\alpha}}"] & (P_i^*,\Lambda \oplus P_i^*) \arrow[dll, out = 0, in = 180, looseness=2, overlay] \\
& (P_i[-1], \Lambda \oplus P_i^*) \arrow[r] & \cdots. & 
\end{tikzcd}
\end{equation*}
We have that $(P_i[-1], \Lambda \oplus P_i^*)\simeq (P_i[-1], \Lambda)\oplus(P_i[-1], P_i^*)$, because the contravariant Hom functor preserves products. Observe now that
$(P_i[-1], \Lambda)=0$, since $P_i[-1]$ and $\Lambda$ are projective modules in different shifted degrees.

It follows that $[\alpha^*]_\alpha$ is surjective on all vertices except $i^*$. For $i^*$, \cref{lemma:EndPi*} implies that all radical endomorphisms factor through $[\alpha^*]_\alpha$, which means that the image of $[\alpha^*]_\alpha$ is $\Rad(P_i^*,\Lambda\oplus P_i^*)$. This completes the proof of the lemma.

\end{proof}

The lemma shows that any element of $(P_i^*, \Lambda\oplus P_i^*)$ is either an isomorphism from $P_i^*$ to itself, or given by some linear combiantion of elements in $\bigoplus\limits_{\alpha:i\rightarrow ?}(P_{t(\alpha)},\Lambda\oplus P_i^*)$, composed with the map $[\alpha^*]_\alpha$. 

This means that any linear combination of paths in $Q^{i^*}$ from some vertex $k$ to $i^*$ is given by some linear combination of paths $k\rightarrow t(\alpha)$ in $Q$ (for some subset of the arrows $\alpha$), each followed by the corresponding arrow $\alpha^*$. 

From the lemma we also see that $\Ker[\alpha^*]_{\alpha} = \Img[\alpha]_\alpha$, meaning that an element of $(\bigoplus\limits_{\alpha: i\rightarrow ?} P_{t(\alpha)}, \Lambda \oplus P_i^*)$ is sent to zero by $[\alpha^*]_\alpha$ if and only if it comes from an element of $(P_i, \Lambda\oplus P_i^*)$. So the composition $[\alpha^*]_\alpha\cdot[\alpha]_\alpha=\sum_{\alpha}\alpha^*\alpha$ is zero, which defines a relation $i\dashrightarrow i^*$ in $Q^{i^*}$, and any relation in $Q^{i^*}$ ending in $i^*$ factors through this relation.

We summarize these observations in the following remark.

\begin{remark}\label{remark:intoi}
There is a one-to-one correspondence between arrows out of $i$ in $Q$ and arrows into $i^*$ in $Q^{i^*}$.
There is a minimal relation $i\dashrightarrow i^*$ in $Q^{i^*}$ given by $\sum_{\alpha}\alpha^*\alpha=0$, and any relation ending in $i^*$ factors through this relation. 

\end{remark}

Before we state the next lemma, which describes the structure of $(\Lambda, P_i^*)$, we will define some morphisms which we will need in order to do so. Recall from \cref{section:setup} that if $r$ is a relation and $\alpha$ is an arrow, both starting in $i$, then $\sfrac{r}{\alpha}$ denotes the linear combination of paths in $Q$ given by taking all paths in $r$ which begin with the arrow $\alpha$, and removing $\alpha$ from them. So $\sfrac{r}{\alpha}$ is a linear combination of paths from vertex $t(\alpha)$ to vertex $t(r)$.
As usual, $\sfrac{r}{\alpha}$ will for simplicity also refer to the induced maps $P_{t(r)}\xrightarrow{\sfrac{r}{\alpha}} P_{t(\alpha)}$ and $(\Lambda, P_{t(r)})\xrightarrow{\sfrac{r}{\alpha}} (\Lambda, P_{t(\alpha)})$.

If we take the mutation triangle
\begin{equation*}
P_i^*\xrightarrow{[\alpha^*]_{\alpha}} \bigoplus\limits_{\alpha: i\rightarrow ?} P_{t(\alpha)} \xrightarrow{[\alpha]_{\alpha}} P_i \rightarrow P_i^*[1]
\end{equation*}
and apply the functor $(P_{t(r)}, - )$ to it, we get a long exact sequence containing the following part
\begin{equation*}
\begin{tikzcd}[column sep = small, row sep = tiny]
\cdots \arrow[r] & (P_{t(r)}, P_i[-1]) \arrow[r] & (P_{t(r)}, P_i^*) \arrow[r, "{[\alpha^*]_{\alpha}}"]  & \bigoplus\limits_{\alpha:i\rightarrow ?}(P_{t(r)}, P_{t(\alpha)})  \arrow[r, "{[\alpha]_{\alpha}}"] & (P_{t(r)}, P_i) \arrow[r] & \cdots, \\
 & 0 \arrow[u, equals] & & &
\end{tikzcd}
\end{equation*}
where $(P_{t(r)}, P_i[-1]) = 0$ because the modules are projective in different shifted degrees. Hence, by exactness $[\alpha^*]_\alpha$ is a monomorphism. 
Notice that for fixed $r$, $[\sfrac{r}{\alpha}]_{\alpha}$ is an element of $\bigoplus\limits_{\alpha:i\rightarrow ?}(P_{t(r)}, P_{t(\alpha)})$, and when composed with $[\alpha]_\alpha$ this gives $[\sfrac{r}{\alpha}]_\alpha\cdot[\alpha]_\alpha=\sum_\alpha \sfrac{r}{\alpha}\circ\alpha=r$. 
Since $r$ is a relation, and hence equal to zero in $\Lambda$, this means that $[\sfrac{r}{\alpha}]_\alpha\in\Ker([\alpha]_\alpha)$. And since $\Ker([\alpha]_\alpha)\simeq\Img([\alpha^*]_\alpha)$ by exactness, we can find an element of $(P_{t(r)}, P_i^*)$ which is sent to $[\sfrac{r}{\alpha}]_\alpha$ by $[\alpha^*]_\alpha$. We call this element $\overline{r}$. In other words, we have that $\overline{r}\in (P_{t(r)}, P_i^*)$, and $\overline{r}$ satisfies $\overline{r}\cdot \alpha^*=\sfrac{r}{\alpha}$ for all $\alpha$. 

We are now ready to state the lemma.

\begin{lemma}\label{lemma:homLambdaPi}
The map
$$(\Lambda, [\overline{r}]_r)\colon  \bigoplus\limits_{r:i\dashrightarrow ?}\big(\Lambda, P_{t(r)}\big) \rightarrow \big(\Lambda, P_{i}^*\big)$$
is surjective. Moreover its kernel coincides with the kernel of $\big( \Lambda, [\sfrac{r}{\alpha}]_{r,\alpha} \big)$.

\end{lemma}

\begin{proof}
We again consider the mutation triangle
\begin{equation*}\label{eq:mutationTriangle}
P_i^*\xrightarrow{[\alpha^*]_{\alpha}} \bigoplus\limits_{\alpha: i\rightarrow ?} P_{t(\alpha)} \xrightarrow{[\alpha]_{\alpha}} P_i \rightarrow P_i^*[1],
\end{equation*}
and we apply the functor $(\Lambda, -)$. Like before, this gives a long exact sequence, containing the following part
\begin{equation*}
\begin{tikzcd}[column sep = small]
\cdots \arrow[r] & (\Lambda, P_i[-1]) \arrow[r] & (\Lambda, P_i^*) \arrow[r, "{[\alpha^*]_{\alpha}}"]  & \bigoplus\limits_{\alpha:i\rightarrow ?}(\Lambda, P_{t(\alpha)})  \arrow[r, "{[\alpha]_{\alpha}}"] & (\Lambda, P_i) \arrow[r] & \cdots,
\end{tikzcd}
\end{equation*}
where $(\Lambda, P_i[-1])=0$ because the modules are projective and in different shifted degrees. So we get that the following sequence is exact
\begin{equation*}
\begin{tikzcd}
0 \arrow[r] & (\Lambda, P_i^*) \arrow[r, "{[\alpha^*]_{\alpha}}"]  & \bigoplus\limits_{\alpha:i\rightarrow ?}(\Lambda, P_{t(\alpha)})  \arrow[r, "{[\alpha]_{\alpha}}"] & (\Lambda, P_i).
\end{tikzcd}
\end{equation*}
By applying $(\Lambda, - )$ to the projective resolution given in \cref{lemma:projres}, we also get that the sequence
\begin{equation*}
\begin{tikzcd}
\Ker([\sfrac{r}{\alpha}]_{r,\alpha}) \arrow[r] & \bigoplus\limits_{r:i\dashrightarrow ?}(\Lambda, P_{t(r)}) \arrow[r, "{[\sfrac{r}{\alpha}]_{r, \alpha}}"]  & \bigoplus\limits_{\alpha:i\rightarrow ?}(\Lambda, P_{t(\alpha)})  \arrow[r, "{[\alpha]_{\alpha}}"] & (\Lambda, P_i)
\end{tikzcd}
\end{equation*}
is exact.

Combining these two exact sequences, we get the following diagram
\begin{equation*}
\begin{tikzcd}
0 \arrow[r] & (\Lambda, P_i^*) \arrow[r, "{[\alpha^*]_{\alpha}}"]  & \bigoplus\limits_{\alpha:i\rightarrow ?}(\Lambda, P_{t(\alpha)})  \arrow[r, "{[\alpha]_{\alpha}}"] & (\Lambda, P_i)\\
\Ker([\sfrac{r}{\alpha}]_{r,\alpha}) \arrow[r] \arrow[u, "0"] & \bigoplus\limits_{r:i\dashrightarrow ?} (\Lambda, P_{t(r)}) \arrow[r, "{[\sfrac{r}{\alpha}]_{r,\alpha}}"] \arrow[u, "{[\overline{r}]_r}"] & \bigoplus\limits_{r:i\dashrightarrow ?} (\Lambda, P_{t(\alpha)}) \arrow[u, equals] \arrow[r, "{[\alpha]_\alpha}"]& (\Lambda, P_i). \arrow[u, equals]
\end{tikzcd}
\end{equation*}
Here, the rows are exact and the squares commute. The first vertical map is zero, and hence an epimorphism, and the equalities are both monomorphisms and epimorphisms. Thus, by the four lemma, the map $[\overline{r}]_r$ is an epimorphism, and hence surjective.

Now, let's consider the relations in $(\Lambda, P_i^*)$, that is, elements in $(\Lambda, P_i^*)$ which are equivalent to zero.
Because $[\alpha^*]_\alpha$ is a monomorphism, an element of $(\Lambda, P_i^*)$ is equivalent to zero if and only if it is in $\Ker([\alpha^*]_\alpha)$.
Any element of $(\Lambda,P_i^*)$ is equal to some element of $\bigoplus\limits_{r:i\dashrightarrow ?}\big(\Lambda, P_{t(r)}\big)$ composed with $[\overline{r}]_{r}$, since $[\overline{r}]_r$ is a surjection.
In particular, this means that any element of $\Ker([\alpha^*]_\alpha)$ is generated by some element of $\bigoplus\limits_{r:i\dashrightarrow ?}\big(\Lambda, P_{t(r)}\big)$, which by commutativity must be sent to zero by $[\sfrac{r}{\alpha}]_{r,\alpha}$. Hence, $\Ker([\sfrac{r}{\alpha}]_{r,\alpha})$ generates $\Ker([\alpha^*]_\alpha)$, and consequently it generates all relations in $(\Lambda, P_i^*)$.
\end{proof}

This lemma shows that any morphism $P_k\rightarrow P_i^*$ can be written as some morphsim  $P_k \rightarrow \bigoplus\limits_{r:i\dashrightarrow ?} P_{t(r)}$ composed with $[\overline{r}]_{r}$. In terms of the quivers, this means that any path $i^*\rightarrow k$ is given as $\gamma\overline{r}$, where $\gamma$ is some path $t(r)\rightarrow k$ and $\overline{r}$ is the arrow $i^*\rightarrow t(r)$ satisfying $\overline{r}\alpha^*=\sfrac{r}{\alpha}$ for each arrow $\alpha\colon i\rightarrow ?$.

The lemma also shows that for a vertex $l$ in $Q$, any relation $P_l\dashrightarrow P_i^*$ is of the form $[\varepsilon_r]_r\cdot[\overline{r}]_r = 0$, for certain $\varepsilon_r\colon P_l\rightarrow P_{t(r)}$. 

Since $[\alpha^*]_\alpha$ is a monomorphism, we get that $[\varepsilon_r]_r\cdot[\overline{r}]_r = 0$ if and only if $[\varepsilon_r]_r\cdot[\overline{r}]_r\cdot[\alpha^*]_\alpha = 0$. This can be rewritten as
$$[\varepsilon_r]_r\cdot[\overline{r}]_r\cdot[\alpha^*]_\alpha = \left(\sum\limits_{r\in R} \varepsilon_r \overline{r} \right)\cdot [\alpha^*]_\alpha = \left[\sum\limits_{r\in R} \varepsilon_r \overline{r}\alpha^* \right]_\alpha = \left[\sum\limits_{r\in R} \varepsilon_r \sfrac{r}{\alpha} \right]_\alpha =0,$$
which means that $[\varepsilon_r]_r\cdot[\overline{r}]_r = 0$ in $(\Lambda, P_i^*)$ if and only if $\sum_r \varepsilon_r\sfrac{r}{\alpha}=0$ for each $\alpha$ starting in $i$. Thus, $[\varepsilon_r]_r\cdot[\overline{r}]_r$ defines a relation in $(\Lambda, P_i^*)$ if and only if $[\varepsilon_r]_r\cdot[\sfrac{r}{\alpha}]_r$ defines a relation in $(\Lambda, P_{t(\alpha)})$ for each $\alpha$. 

In terms of the quivers, this means $\sum\limits_{r\in R} \varepsilon_r \overline{r}=0$ will define a relation in $Q^{i^*}$ if and only if $\sum\limits_{r\in R} \varepsilon_r \sfrac{r}{\alpha} =0$ defines a relation in $Q$ for each $\alpha\colon i\rightarrow ?$.

\medskip

\begin{remark}\label{remark:outofi}
There is a one-to-one correspondence between arrows in $Q^{i^*}$ starting in $i^*$ and minimal relations in $Q$ starting in $i$, given by $\overline{r}\leftrightarrow r$. 
This means that there are no minimal relations in $Q^{i^*}$ from $i$ to itself. Because if $r$ is any relation in $Q$ starting in $i$, then 
$$r=\sum\limits_{\alpha:i\rightarrow ?} \sfrac{r}{\alpha} \alpha=\sum\limits_{\alpha:i\rightarrow ?}\overline{r}\alpha^*\alpha=\overline{r}\sum\limits_{\alpha:i\rightarrow ?} \alpha^*\alpha,$$
so $r$ factors through $\sum_\alpha \alpha^* \alpha$ in $Q^{i^*}$, and hence is not minimal.

Also, any linear combination of paths in $Q^{i^*}$ defines a relation $i^*\dashrightarrow l$ if and only if precomposing it with each of the arrows $\alpha^*$ gives a linear combination of paths which is equal to a relation in $Q$. 
\end{remark}

\bigskip

We are now ready to define the quiver $Q^{i^*}$, whose path algebra (by construction) is isomorphic to $\End_\Lambda(\Lambda\oplus P_i^*)^{\text{op}}$.

\bigskip

\begin{lemma}\label{lemma:Qi*}
Let $(Q^{i^*}, I^{i^*})$ be the following quiver with relations:
\begin{enumerate}\addtocounter{enumi}{-1}
\item The set of vertices in $Q^{i^*}$ consists of all vertices in $Q$, and an additional vertex $i^*$.
\item All arrows in $Q$ are also arrows in $Q^{i^*}$, and all relations in $Q$ are also relations in $Q^{i^*}$.
\item There is an arrow $\alpha^*\colon t(\alpha)\rightarrow i^*$ in $Q^{i^*}$ for each arrow $\alpha\colon i\rightarrow t(\alpha)$ in $Q$, and all arrows ending in $i^*$ are given as such.
\item There is an arrow $\overline{r}\colon i^*\rightarrow t(r)$ for each minimal relation $r\colon i \dashrightarrow t(r)$ in $Q$, and all arrows starting in $i^*$ are given as such.
\item There is a relation $i\dashrightarrow i^*$ in $Q^{i^*}$, given by $\sum\limits_{\alpha:i\rightarrow ?}\alpha^*\alpha=0$, and no other minimal relations end in $i^*$.
\item For each arrow $\alpha$ starting in $i$, and each relation $r$ starting in $i$, there is a relation $t(\alpha)\dashrightarrow t(r)$ in $Q^{i^*}$, given by $\overline{r}\alpha^*=\sfrac{r}{\alpha}$.
\item There is a relation defined by $\sum\limits_{r:i\dashrightarrow ?} \varepsilon_r \overline{r}=0$ in $Q^{i^*}$ if and only $\sum\limits_{r:i\dashrightarrow ?} \varepsilon_r \sfrac{r}{\alpha}=0$ defines a relation in $Q$ for each arrow $\alpha\colon i\rightarrow ?$.
\end{enumerate}
Then the path algebra $kQ^{i^*}/I^{i^*}$ is isomorphic to $\End_\Lambda(\Lambda\oplus P_i^*)^{\text{op}}$.
\end{lemma}

\begin{proof}
Clearly, the primitive idempotents in $\End_\Lambda(\Lambda\oplus P_i^*)^{\text{op}}$ are the same as those in $\End_\Lambda(\Lambda)^{\text{op}}$, plus the one in $\End_\Lambda(P_i^*)^{\text{op}}$. Hence, $Q^{i^*}$ has the same vertices as $Q$, plus the additional vertex $i^*$. This proves point $0$.

Since $\End_\Lambda(\Lambda)\subseteq \End_\Lambda(\Lambda\oplus P_i^*)$, any arrow and relation in $Q$ will still exist in $Q^{i^*}$, proving point $1$.

From \cref{remark:intoi} we know that there is a one-to-one correspondence between arrows in $Q$ starting in $i$, and arrows in $Q^{i^*}$ ending in $i^*$, which proves point $2$. The remark also tells us that there is one, and only one, minimal relation ending in $i^*$, namely $\sum\limits_{\alpha:i\rightarrow ?}\alpha^*\alpha=0$. This proves point $4$. 

\Cref{remark:outofi} tells us that $\overline{r} \leftrightarrow r$ defines a one-to-one correspondence between arrows in $Q^{i^*}$ starting in $i^*$ and relations in $Q$ starting in $i$, thus proving point $3$. 
By construction, the arrows $\overline{r}$ in $Q^{i^*}$ satisfy $\overline{r}\alpha^*=\sfrac{r}{\alpha}$ for each arrow $\alpha$ in $Q$ starting in $i$. This defines a set of relations on $Q^{i^*}$, which proves point $5$.
Also from \cref{remark:outofi}, we get that a linear combination of paths in $Q^{i^*}$ defines a relation $i^*\dashrightarrow l$ if and only if precomposing it with each of the arrows $\alpha^*$ gives a linear combination of paths which is equal to a relation in $Q$. In other words, we have that
$$\sum\limits_{r:i\dashrightarrow ?} \varepsilon_r \overline{r}=0 \text{ in } Q^{i^*} \quad \iff \quad \sum\limits_{r:i\dashrightarrow ?} \varepsilon_r \sfrac{r}{\alpha}=0 \text{ for all } \alpha\colon i\rightarrow ? \text{ in } Q,$$
which proves point $6$.

This concludes the proof of the lemma.

\end{proof}

Now we are ready to prove the main result, namely that mutating a quiver according to the mutation procedure stated in \cref{section:mutationsteps} yields a quiver whose path algebra is isomorphic to to the right tilting mutation of the path algebra of the original quiver.
More precisely, the statement is (in the notation of \cref{theorem:mainResult}) that $m_i(Q)\simeq Q_i^{{i^*}}$, where $m_i(Q)$ is the quiver obtained by mutating $Q$ at vertex $i$, and $Q_i^{i^*}$ is the quiver obtained by removing vertex $i$ from the quiver $Q^{i^*}$. 
And by combining \cref{lemma:Qi*} and \cref{lemma:QuiverRemoveVertex}, we know that the path algebra of $Q_i^{i^*}$ is isomorphic to $\End_\Lambda(\mu_i^R(\Lambda))^{\text{op}}$, so this implies that the path algebra of $m_i(Q)$ is isomorphic to the right tilting mutation of $\Lambda$ at $P_i$.

\begin{proof}[Proof of Theorem \ref{theorem:mainResult}]
We want to show that $m_i(Q)\simeq Q_i^{i^*}$. The structure of $m_i(Q)$ is defined in \cref{section:mutationsteps}, so we already know it. Hence, we must determine the structure of $Q_i^{i^*}$, and show that it is the same as the known structure of $m_i(Q)$.
In order to do so, we will apply \cref{lemma:QuiverRemoveVertex} to the quiver obtained in \cref{lemma:Qi*}. Recall that $\mu_i^R(\Lambda)\simeq\Lambda/P_i\oplus P_i^*$, so we have that 
$$\End_\Lambda(\mu_i^R(\Lambda))^{\text{op}}\simeq\End_\Lambda(\Lambda/P_i\oplus P_i^*)^{\text{op}}.$$
Starting with a quiver $Q$ whose endomorphism ring is the algebra $\Lambda$, \cref{lemma:Qi*} yields the quiver $Q^{i^*}$ whose endomorphism ring is isomorphic to $\End_\Lambda(\Lambda\oplus P_i^*)$. 
What we want to do now is apply \cref{lemma:QuiverRemoveVertex} to $Q^{i^*}$, but in order to do so we must check that the assumptions in the lemma are satisfied. The assumptions we need to check are that there are no cycles of length $1$ on vertex $i$, and that there are no minimal relations from $i$ to itself.

By assumption $i$ has no cycles of length $1$ in $Q$, and because we don't add any such cycles in \cref{lemma:Qi*}, the same is true for $Q^{i^*}$. Also, we know from \cref{remark:outofi} that $Q^{i^*}$ has no minimal relations from $i$ to itself.
Thus, the assumtions needed for using \cref{lemma:QuiverRemoveVertex} are satisfied.

Now, applying \cref{lemma:QuiverRemoveVertex} to the quiver $Q^{i^*}$ will then yield the quiver $Q_i^{i^*}$, corresponding to $\End_\Lambda(\Lambda/P_i\oplus P_i^*)$, which is precisely what we want. 
As we can see from \cref{lemma:QuiverRemoveVertex}, most parts of the quiver $Q^{i^*}$ will remain unchanged when we transform it into $Q_i^{i^*}$. It keeps all vertices except $i$, as well as all arrows and relations which neither start nor end in $i$. Basically the only difference between the quivers is that vertex $i$ is removed, that minimal paths through $i$ are replaced with arrows between the same vertices, and that relations starting or ending in $i$ are extended by one step in the necessary direction. 

Note that the assumption that there are no arrows $i\rightarrow i$ in $Q$ ensures that there are no arrows $i\rightarrow i^*$ in $Q^{i^*}$, which means that none of the arrows ending in $i^*$ will be removed in $Q_i^{i^*}$. 

Looking at \cref{lemma:Qi*}, we can go through each step and make the necessary change to $Q^{i^*}$ as implied by \cref{lemma:QuiverRemoveVertex}, ending up with the quiver $Q_i^{i^*}$ having the following structure.
\begin{enumerate}\addtocounter{enumi}{-1}
\item The set of vertices in $Q^{i^*}_i$ consists of $i^*$ and all vertices in $Q$ except $i$.
\item All arrows in $Q$ which neither start nor end in $i$ are also arrows in $Q_i^{i^*}$. The same is true for relations.
\item There is an arrow $\alpha\beta\colon h\rightarrow j$ in $Q_i^{i^*}$ for each pair of arrows $\beta\colon h \rightarrow i$ and $\alpha\colon i \rightarrow j$ in $Q$.
\item There is an arrow $\alpha^*\colon t(\alpha)\rightarrow i^*$ in $Q_i^{i^*}$ for each arrow $\alpha\colon i\rightarrow t(\alpha)$ in $Q$, and all arrows ending in $i^*$ are given as such. 
\item There is an arrow $\overline{r}\colon i^*\rightarrow t(r)$ in $Q_i^{i^*}$ for each minimal relation $r\colon i \dashrightarrow t(r)$ in $Q$, and all arrows starting in $i^*$ are given as such.
\item For each arrow $\beta\colon h\rightarrow i$ in $Q$ there is a relation $h\dashrightarrow i^*$ in $Q_i^{i^*}$, given by $\sum\limits_{\alpha:i\rightarrow ?}\alpha^*\alpha\beta=0$. 
\item For each arrow $\alpha$ in $Q$ starting in $i$, and each relation $r$ starting in $i$, there is a relation $t(\alpha)\dashrightarrow t(r)$ in $Q^{i^*}$, given by $\overline{r}\alpha^*=\sfrac{r}{\alpha}$.
\item If $r\colon h\dashrightarrow i$ is a relation in $Q$, and $B$ is the set of arrows ending in $i$ (which means that $r$ can be written as $\sum\limits_{\beta\in B} \beta \bsfrac{r}{\beta}=0$), then there is a relation $h\dashrightarrow t(\alpha)$ in $Q_i^{i^*}$ for each arrow $\alpha$, defined by $\sum\limits_{\beta\in B} \alpha\beta \bsfrac{r}{\beta}=0$.
\item There is a relation defined by $\sum\limits_{r:i\dashrightarrow ?} \varepsilon_r \overline{r}=0$ in $Q_i^{i^*}$ if and only $\sum\limits_{r:i\dashrightarrow ?} \varepsilon_r \sfrac{r}{\alpha}=0$ defines a relation in $Q$ for each arrow $\alpha\colon i\rightarrow ?$.
\end{enumerate}

As we can see, this gives precisely the same quiver as the mutation procedure we defined in section \cref{section:mutationsteps}, which means that $Q_i^{i^*}\simeq m_i(Q)$. Hence, mutating a quiver $Q$ at vertex $i$ according to the mutation procedure yields a quiver whose path algebra is isomorphic to the right tilting mutation of the path algebra of $Q$ at the indecomposable projective summand $P_i$. This concludes the proof of the main theorem.

\end{proof}

\vfill

\section{Example}\label{section:example}
Let us now consider an example to see how the mutation procedure can be used in practice. This example is inspired by Ladkani's work in \cite{Ladkani2013}, specifically the following result. Let $k$ be a field, and let $A_{n,m}$ denote the algebra obtained by taking the path algebra over $k$ of the linear quiver $A_{n}$ modulo the ideal generated by all paths of length $m$. Then the algebra $A_{r\cdot n ,r+1}$ is derived equivalent to $kA_r\otimes_k kA_n$ (which is isomorphic to the path algebra of a commutative $r\times n$-rectangle).

We will now show how we can use our mutation procedure to not only recreate this result, but also give an explicit series of derived equivalencies between the two algebras.
Take the quiver corresponding to the algebra $A_{2n,3}$, which is the line quiver $A_{2n}$ with a relation for every composition of 3 consecutive arrows. We will now show that by preforming a series of mutations, we can transform this quiver into the quiver given by a column of connected commutative squares, with relations for every com position of two consecutive arrows (this is essentially the quiver of $kA_2 \otimes_k kA_n$). Note, throughout this section we will occasionally renumber the vertices in the quivers, in order to highlight the relevant part of the quiver for each mutation.
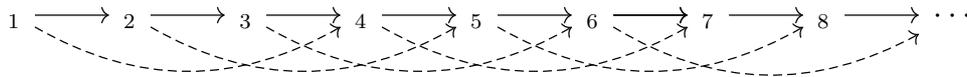
\begin{equation*}
\begin{tikzcd}
_1 \arrow[r] \arrow[rrr, bend right, dashed] & _2 \arrow[r] \arrow[rrr, bend right, dashed] & _3 \arrow[r] \arrow[rrr, bend right, dashed] & _4 \arrow[r] \arrow[rrr, bend right, dashed] & _5 \arrow[r] \arrow[rrr, bend right, dashed] & _6 \arrow[r] \arrow[r] \arrow[rrr, bend right, dashed] & _7 \arrow[r] & _8 \arrow[r]  & \cdots
\end{tikzcd}
\end{equation*}
We begin by mutating at the first (leftmost) vertex. 
To do so, we apply the mutation procedure at vertex $1$. Since there are no arrows or relations into vertex $1$, we see that steps 1, 3 and 6 don't apply in this case. By step 2 we flip the arrow $1\rightarrow 2$ to an arrow $1^*\leftarrow 2$, and by step 4 the relation $1\dashrightarrow 4$ becomes an arrow $1^*\rightarrow 4$. Thus we get a new composition through the vertex $1^*$, and by step 5 we add a relation corresponding to that composition. this relation is what makes the commutative square. Finally, by step 7 we get a relation $1^*\dashrightarrow 5$, since there is a zero relation $2\dashrightarrow 5$ which is hit by the relation $1\dashrightarrow 4$. Thus, we end up with the following.
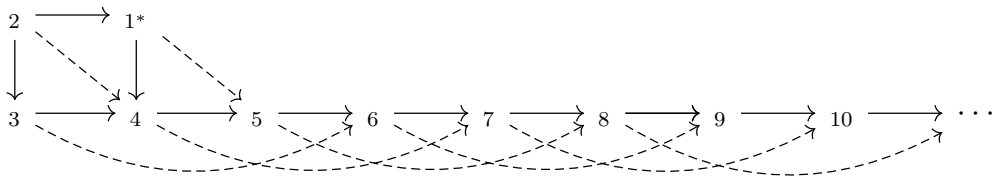
\begin{equation*}
\begin{tikzcd}
_2 \arrow[d] \arrow[r] \arrow[dr, dashed] & _{1^*} \arrow[d] \arrow[dr, dashed] & & & & & & &\\
_3 \arrow[r] \arrow[rrr, bend right, dashed] & _4 \arrow[r] \arrow[rrr, bend right, dashed] & _5 \arrow[r] \arrow[rrr, bend right, dashed] & _6 \arrow[r] \arrow[rrr, bend right, dashed] & _7 \arrow[r] \arrow[rrr, bend right, dashed] & _8 \arrow[r] \arrow[r] \arrow[rrr, bend right, dashed] & _9 \arrow[r] & _{10} \arrow[r] & \cdots  &
\end{tikzcd}
\end{equation*}
We proceed by mutating at vertex $3$. By mutation step $1$ we get a new arrow $2\rightarrow 4$, corresponding to the composition through $3$, and this arrow is cancelled against the commutativity relation $2\dashrightarrow 4$. Step $2$ works just as in the previous mutation, by flipping the arrow $3\rightarrow 4$, and step $3$ says we should add a relation $2\dashrightarrow 3$ corresponding to the composition of these two arrows. Steps $4, 5$ and $7$ work just as in the previous mutation, and step $6$ isn't used since there are no relations ending in vertex $3$.

Thus, by mutating the above quiver at vertex $3$ according to the mutation rules, we get the following quiver. 
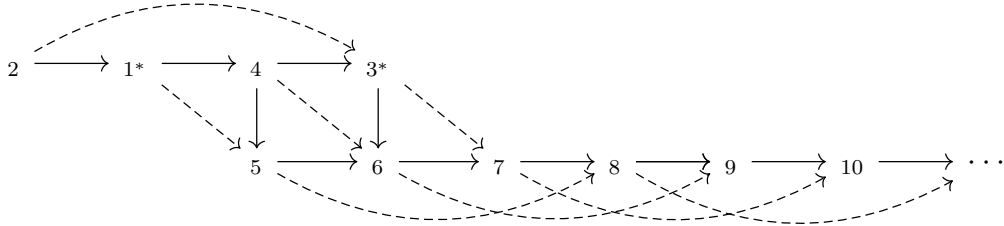
\begin{equation*}
\begin{tikzcd}
_2  \arrow[r] \arrow[rrr, dashed, bend left] & _{1^*} \arrow[r] \arrow[dr, dashed] & _4 \arrow[d] \arrow[r] \arrow[dr, dashed] & _{3^*} \arrow[d] \arrow[dr, dashed] & & & & &\\
 &  & _5 \arrow[r] \arrow[rrr, bend right, dashed] & _6 \arrow[r] \arrow[rrr, bend right, dashed] & _7 \arrow[r] \arrow[rrr, bend right, dashed] & _8 \arrow[r] \arrow[r] \arrow[rrr, bend right, dashed] & _9 \arrow[r] & _{10} \arrow[r] & \cdots & 
\end{tikzcd}
\end{equation*}
As we can see, the resulting quiver is more or less equal to the one we started with, except the commutative square has ''moved'' two steps to the right. 
Next, we mutate at vertex $5$. Notice that locally, the quiver around vertex $5$ is quite similar to the quiver around vertex $3$ before the last mutation. The only meaningful difference is that there is a relation ending in $5$. This means that each step in the mutation is the same as in the previous mutation, except step $6$. Mutation step $6$ tells us that the relation $1^*\dashrightarrow 5$ is extended to a relation $1^*\dashrightarrow 6$. 
Thus, mutating the above quiver at vertex $5$ results in the following quiver.
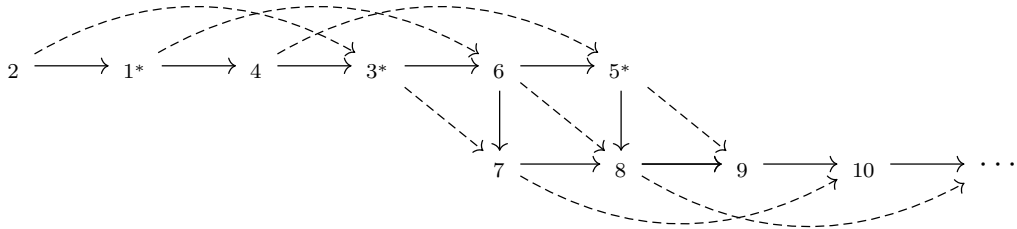
\begin{equation*}
\begin{tikzcd}
_2  \arrow[r] \arrow[rrr, dashed, bend left] & _{1^*} \arrow[r] \arrow[rrr, bend left, dashed] & _4  \arrow[r] \arrow[rrr, bend left, dashed] & _{3^*} \arrow[r] \arrow[dr, dashed] & _6 \arrow[d] \arrow[r] \arrow[dr, dashed] & _{5^*} \arrow[d] \arrow[dr, dashed] & & &\\
 & & & & _7 \arrow[r] \arrow[rrr, bend right, dashed] & _8 \arrow[r] \arrow[r] \arrow[rrr, bend right, dashed]& _9 \arrow[r] & _{10} \arrow[r] & \cdots & 
\end{tikzcd}
\end{equation*} 
Agian, we see that the mutation yields a quiver which is similar, but where the commutative square has been moved two steps to the right.
By using the mutation procedure, it is easy to check that this pattern continues, i.e. that mutating at vertex $7$ will now move the square two more steps to the right. For the remainder of this example, we will no longer explicitly mention which step of the mutation procedure contributes what to a mutated quiver. We will simply show a quiver, state which vertex we mutate at, and then show the resulting mutated quiver.

If we take the last quiver and keep mutating at the vertex in the lower left corner of the commutative square, we can push the square as far to the right as we want.
Eventually, the commutative square is pushed all the way to the end of the quiver, and using the mutation procedure it is easy to check that nothing strange happens with the square at the right end of the quiver. 

We end up with the following quiver (recall that we are looking at $A_{2n,3}$).
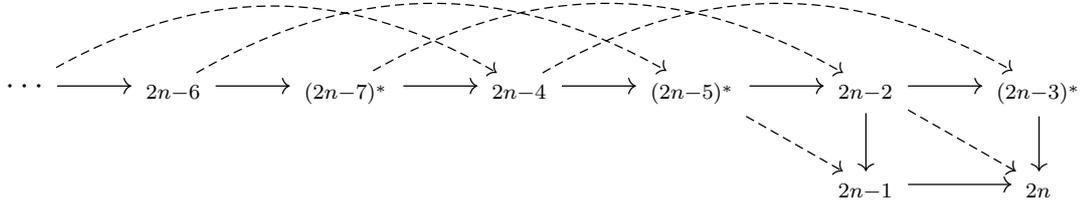
\begin{equation*}
\begin{tikzcd}
\cdots \arrow[r] \arrow[rrr, bend left,  dashed]  & _{2n-6} \arrow[r] \arrow[rrr, bend left, dashed] & _{(2n-7)^*} \arrow[r] \arrow[rrr, bend left, dashed] & _{2n-4} \arrow[r] \arrow[rrr, bend left, dashed] & _{(2n-5)^*} \arrow[r] \arrow[dr, dashed] & _{2n-2} \arrow[d] \arrow[r] \arrow[dr, dashed] & _{(2n-3)^*} \arrow[d] \\
& & & & & _{2n-1} \arrow[r] & _{2n} 
\end{tikzcd}
\end{equation*}
Since the entire left part of this quiver (everything but the commutative square on the end) has the same form as the quiver we started with, we can repeat the same construction to push another commutative square along the quiver. When the new square gets close to the existing square at the right end of the quiver, it is not clear what will happen when we mutate. Mutating at a vertex doesn't affect parts of the quiver that are sufficiently far away (in terms of arrows and relations), but when the two squares are close they might interact in unexpected ways. 
So let's see what happens. 

Pushing the commutative square to the right as descirbed above, eventually the quiver will look like this (note that we have renumbered the vertices)
 
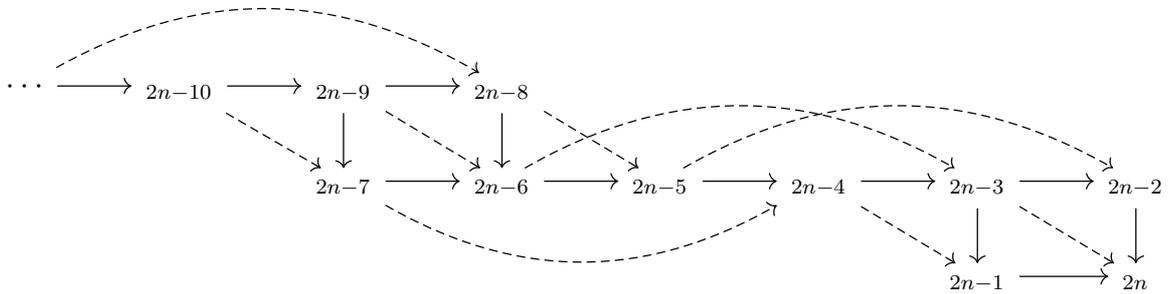
\begin{equation*}
\begin{tikzcd}
\cdots \arrow[r] \arrow[rrr, bend left, dashed] & _{2n-10} \arrow[r] \arrow[dr, dashed]& _{2n-9} \arrow[r] \arrow[d] \arrow[dr, dashed] & _{2n-8} \arrow[d] \arrow[dr, dashed] & & & & \\
& & _{2n-7} \arrow[r] \arrow[rrr, bend right, dashed] & _{2n-6} \arrow[r] \arrow[rrr, bend left, dashed] & _{2n-5} \arrow[r] \arrow[rrr, bend left, dashed] & _{2n-4} \arrow[r] \arrow[dr, dashed] & _{2n-3} \arrow[d] \arrow[r] \arrow[dr, dashed] & _{2n-2} \arrow[d] \\
& & & & & & _{2n-1} \arrow[r] & _{2n} \\
& & & & & & &
\end{tikzcd}
\end{equation*}
We mutate at vertex $2n-7$ using the mutation procedure, and see that we get the following quiver
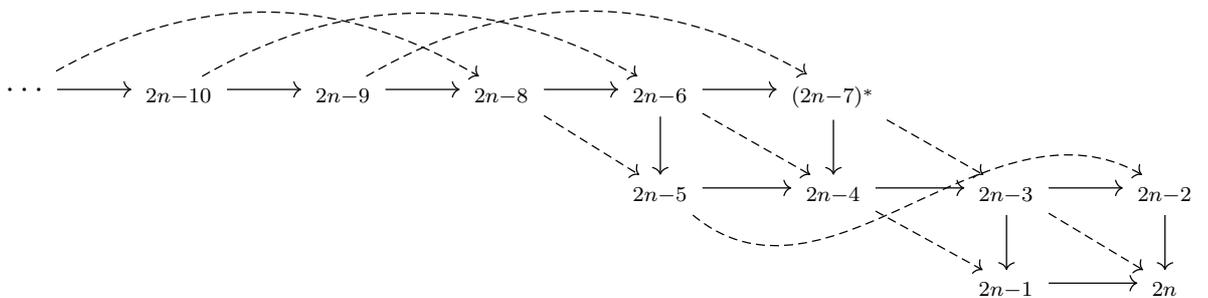
\begin{equation*}
\begin{tikzcd}
\cdots \arrow[r] \arrow[rrr, bend left, dashed] & _{2n-10} \arrow[r] \arrow[rrr, bend left, dashed] & _{2n-9} \arrow[r] \arrow[rrr, bend left, dashed] & _{2n-8} \arrow[r] \arrow[dr, dashed] & _{2n-6} \arrow[d] \arrow[r] \arrow[dr, dashed] & _{(2n-7)^*} \arrow[d] \arrow[dr, dashed] & & \\
& & & & _{2n-5} \arrow[r] \arrow[rrr, out=-40, in=150, dashed] & _{2n-4} \arrow[r] \arrow[dr, dashed] & _{2n-3} \arrow[d] \arrow[r] \arrow[dr, dashed] & _{2n-2} \arrow[d] \\
& & & & & & _{2n-1} \arrow[r] & _{2n} 
\end{tikzcd}
\end{equation*}
Finally, applying the mutation procedure at vertex $2n-5$ gives us
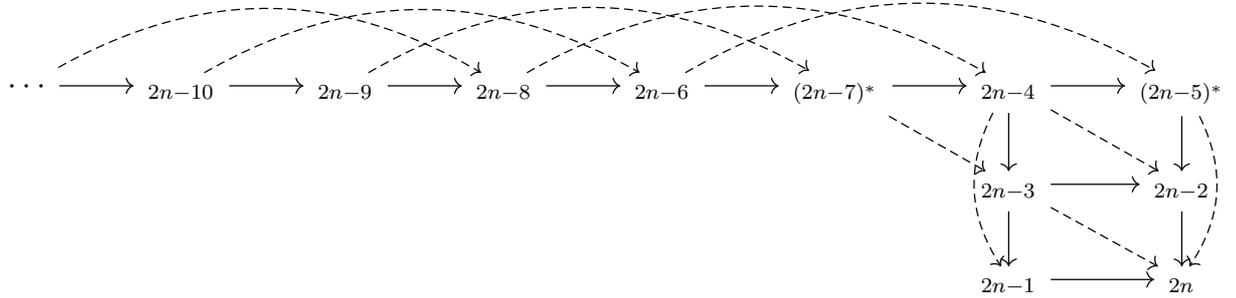
\begin{equation*}
\begin{tikzcd}
\cdots \arrow[r] \arrow[rrr, bend left, dashed] & _{2n-10} \arrow[r] \arrow[rrr, bend left, dashed] & _{2n-9} \arrow[r] \arrow[rrr, bend left, dashed] & _{2n-8} \arrow[r] \arrow[rrr, bend left, dashed] & _{2n-6} \arrow[r] \arrow[rrr, bend left, dashed] & _{(2n-7)^*} \arrow[r] \arrow[dr, dashed] & _{2n-4} \arrow[d] \arrow[r] \arrow[dr, dashed] \arrow[dd, bend right, dashed] & _{(2n-5)^*} \arrow[d] \arrow[dd, bend left, dashed] \\
& & & & & & _{2n-3} \arrow[d] \arrow[r] \arrow[dr, dashed] & _{2n-2} \arrow[d] \\
& & & & & & _{2n-1} \arrow[r] & _{2n} 
\end{tikzcd}
\end{equation*}
So the new commutative square slides in on top of the old one.
Again, the left part of the quiver is like when we started, so as before we can repeat this construction. Now, we only need to check the last mutation, because the process of moving the square to the right is identical to before.
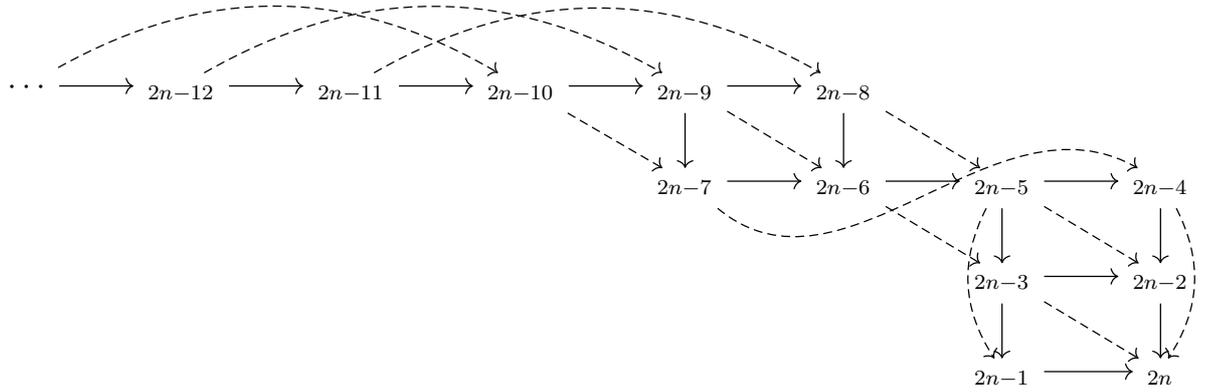
\begin{equation}
\begin{tikzcd}
\cdots \arrow[r] \arrow[rrr, bend left, dashed] & _{2n-12} \arrow[r] \arrow[rrr, bend left, dashed] & _{2n-11} \arrow[r] \arrow[rrr, bend left, dashed] & _{2n-10} \arrow[r] \arrow[dr, dashed] & _{2n-9} \arrow[d] \arrow[r] \arrow[dr, dashed] & _{2n-8} \arrow[d] \arrow[dr, dashed] & & \\
& & & & _{2n-7} \arrow[r] \arrow[rrr, out=-40, in=150, dashed] & _{2n-6} \arrow[r] \arrow[dr, dashed] & _{2n-5} \arrow[d] \arrow[r] \arrow[dr, dashed] \arrow[dd, bend right, dashed] & _{2n-4} \arrow[d] \arrow[dd, bend left, dashed] \\
& & & & & & _{2n-3} \arrow[d] \arrow[r] \arrow[dr, dashed] & _{2n-2} \arrow[d] \\
& & & & & & _{2n-1} \arrow[r] & _{2n}
\end{tikzcd}
\end{equation}
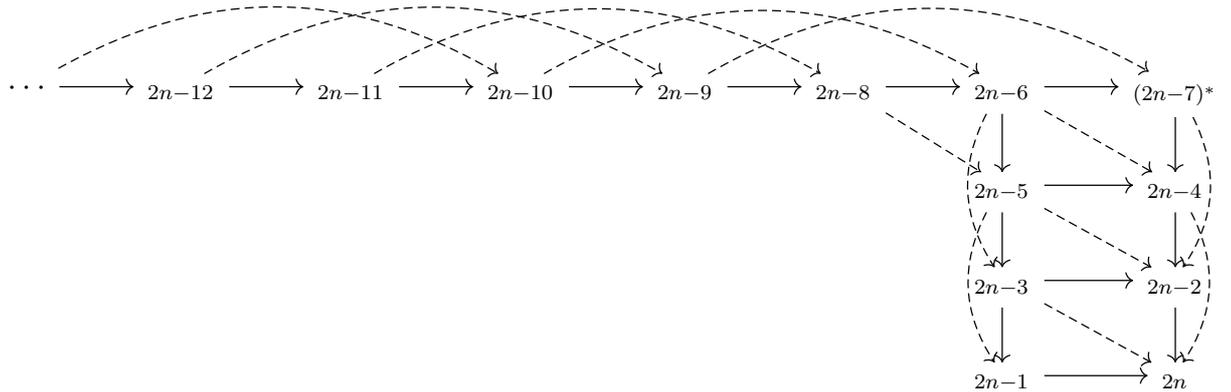
\begin{equation*}
\begin{tikzcd}
\cdots \arrow[r] \arrow[rrr, bend left, dashed] & _{2n-12} \arrow[r] \arrow[rrr, bend left, dashed] & _{2n-11} \arrow[r] \arrow[rrr, bend left, dashed] & _{2n-10} \arrow[r] \arrow[rrr, bend left, dashed] & _{2n-9} \arrow[r] \arrow[rrr, bend left, dashed] & _{2n-8} \arrow[r] \arrow[dr, dashed] & _{2n-6} \arrow[d] \arrow[r] \arrow[dd, bend right, dashed] \arrow[dr, dashed] & _{(2n-7)^*} \arrow[d] \arrow[dd, bend left, dashed] \\
& & & & & & _{2n-5} \arrow[d] \arrow[r] \arrow[dr, dashed] \arrow[dd, bend right, dashed] & _{2n-4} \arrow[d] \arrow[dd, bend left, dashed] \\
& & & & & & _{2n-3} \arrow[d] \arrow[r] \arrow[dr, dashed] & _{2n-2} \arrow[d] \\
& & & & & & _{2n-1} \arrow[r] & _{2n} 
\end{tikzcd}
\end{equation*}
Observe that the mutation doesn't affect the bottom part of the column. This means that, from the perspective of a new square that has been pushed along the quiver to the position right before the column, the situation is identical to before the last mutation.
Because one mutation has a limited range of influence in the quiver, the same is true no matter how tall the "commutative column" is. Hence we can move all the vertices in the quiver into the column part, by first creating a commutative square at the beginning of the quiver, pushing it to the right until it becomes part of the column, and repeat. When we get to the beginning of the quiver, we can see that it all wraps up nicely:
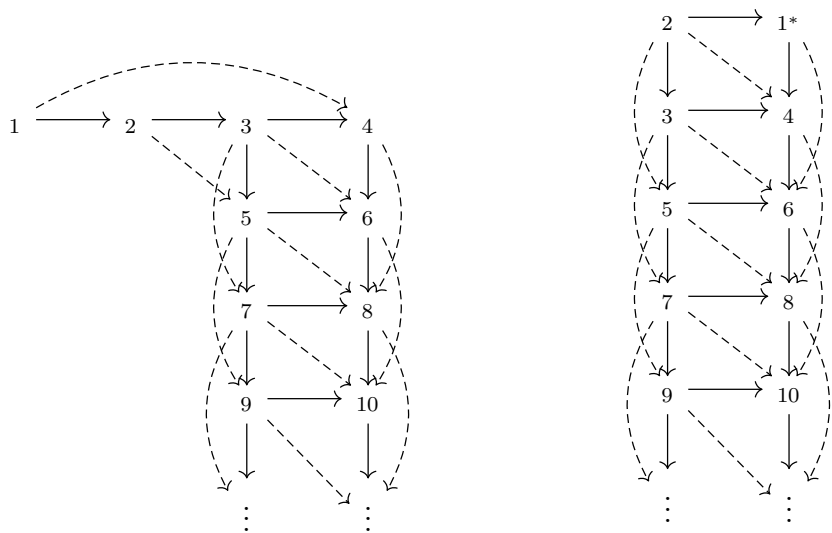
\begin{multicols}{2}
\begin{equation*}
\begin{tikzcd}
 _{1} \arrow[r] \arrow[rrr, bend left, dashed] & _{2} \arrow[r] \arrow[dr, dashed] & _3 \arrow[d] \arrow[r] \arrow[dr, dashed] \arrow[dd, bend right, dashed] & _{4} \arrow[d] \arrow[dd, bend left, dashed] \\
& & _5 \arrow[d] \arrow[r] \arrow[dd, bend right, dashed] \arrow[dr, dashed]  & _{6} \arrow[d] \arrow[dd, bend left, dashed]  \\
  & & _{7} \arrow[d] \arrow[r] \arrow[dr, dashed] \arrow[dd, bend right, dashed] & _{8} \arrow[d] \arrow[dd, bend left, dashed] \\
 & & _{9} \arrow[d] \arrow[r] \arrow[dr, dashed] & _{10} \arrow[d]  \\
 & & \vdots & \vdots
\end{tikzcd}
\end{equation*}
\begin{equation*}
\begin{tikzcd}
 _{2} \arrow[r] \arrow[d] \arrow[dd, bend right, dashed] \arrow[dr, dashed] & _{1^*} \arrow[d] \arrow[dd, bend left, dashed] \\
  _3 \arrow[d] \arrow[r] \arrow[dr, dashed] \arrow[dd, bend right, dashed] & _{4} \arrow[d] \arrow[dd, bend left, dashed] \\
 _5 \arrow[d] \arrow[r] \arrow[dd, bend right, dashed] \arrow[dr, dashed]  & _{6} \arrow[d] \arrow[dd, bend left, dashed]  \\
 _{7} \arrow[d] \arrow[r] \arrow[dr, dashed] \arrow[dd, bend right, dashed] & _{8} \arrow[d] \arrow[dd, bend left, dashed] \\
_{9} \arrow[d] \arrow[r] \arrow[dr, dashed] & _{10} \arrow[d]  \\
 \vdots & \vdots
\end{tikzcd}
\end{equation*}
\end{multicols}
Thus, we have an explicit chain of right tilting mutations which allows us to transform the quiver $A_{n,3}$ into the quiver consisting of a column of commutative squares with relations of length $2$ along the outsides. 
This in turn tells us that their respective path algebras are derived equivalent.

\bigskip

\printbibliography

\Addresses

\end{document}